\documentclass[onefignum,onetabnum]{siamart171218}

\usepackage{xcolor}
\usepackage{amsfonts}
\usepackage{graphicx}
\usepackage{epstopdf}
\usepackage{overpic}
\usepackage{amsmath,esint}
\usepackage{epsfig}
\usepackage{todonotes}
\usepackage{tikz}
\usetikzlibrary{arrows}
\usetikzlibrary{fadings}
\usepackage{pgfplots}
\pgfplotsset{compat=1.10}
\usepgfplotslibrary{fillbetween}
\usetikzlibrary{patterns}
\usepackage{enumerate}
\usepackage{enumitem}
\usepackage[mathscr]{euscript}
\usepackage{algorithmic}
\usepackage{mathtools}

\newsiamremark{remark}{Remark}

\title{Data-driven algorithms for signal processing with trigonometric rational functions\thanks{Funding: This work is partly supported by National Science Foundation grants DMS-2103317, DMS-1818757, DMS-1830274, DMS-2045646, DMS-1952757, and DGE-1650441.}} 
\author{ Heather Wilber\thanks{Oden Institute for Computational and Engineering Sciences, University of Texas at Austin, Austin, TX , 78712-1229, United States (\email{heather.wilber@oden.utexas.edu}).}
\and
Anil Damle\thanks{Department of Computer Science, Cornell University, Ithaca, NY 14853-4201, United States (\email{damle@cornell.edu}).}
\and Alex Townsend\thanks{Department of Mathematics, Cornell University, Ithaca, NY 14853-4201, United States (\email{townsend@cornell.edu}).} } 

\begin{document}
\maketitle

\begin{abstract}
Rational approximation schemes for reconstructing periodic signals from samples with poorly separated spectral content are described. These methods are automatic and adaptive, requiring no tuning or manual parameter selection.  Collectively, they form a framework for fitting trigonometric rational models to data that is robust to various forms of corruption, including additive Gaussian noise, perturbed sampling grids, and missing data. Our approach combines a variant of Prony's method with a modified version of the AAA algorithm. Using representations in both frequency and time space, a collection of algorithms is described for adaptively computing with trigonometric rationals. This includes procedures for differentiation, filtering, convolution, and more. A new MATLAB software system based on these algorithms is introduced. Its effectiveness is illustrated with synthetic and practical examples  drawn from applications including biomedical monitoring, acoustic denoising, and feature detection. \end{abstract}



\section{Introduction}
Recovering functions from noisy, incomplete, or corrupted samples is a ubiquitous task in signal and data processing~\cite{kay1993fundamentals}. Here, we recover underlying signals that contain impulses, shocks, or other algebraic singularities that can cause traditional Fourier-based methods to underperform or fail. Examples where these signals appear include sensor monitoring and event detection tasks in seismology and oceanography~\cite{cichowicz1993automatic, martin2021linear}, biomedical signal processing~\cite{fridli2011rational, gilian2014ecg}, and time evolution along rays in nonsmooth media~\cite{vcerveny1982computation, ricker2012transient}.  We present a novel computing framework based on data-driven approximation with two complementary representations: (1) barycentric trigonometric rational approximations and (2) their Fourier transforms, which take the form of short sums of complex exponentials.  Toggling between these two representations lets us overcome computational and data-related challenges.

\begin{table} 
\centering
\begin{tabular}{cc}
Barycentric form & Exponential sums \\ 
\hline 
Differentiation (closed-form formula)~\cite{berrut2005recent} & Filtering and recompression~\cite{haut2013solving}\\
Imputing missing data~\cite{nakatsukasa2018aaa} & Pole symmetry preservation~\cite{beylkin2009nonlinear} \\
Stable evaluation~\cite{austin2017numerical,higham2004numerical} & Robustness to noise~\cite{potts2013parameter}\\
Rootfinding~\cite{nakatsukasa2018aaa}, identifying extrema & convolution~\cite{potts2013parameter}, cross-correlations \\
\hline
\end{tabular} 
\caption{Operations that are efficient and robust in the two representations. Having both representations and toggling between them allows us to compute a range of operations in signal processing.}
\label{tab:Operations} 
\end{table}

Several families of functions, including rationals, wavelets, and radial basis functions, are well-suited for resolving sharp features in data and modeling phenomena with slow-decaying spectral content~\cite{trefethen2013approximation, unser2000wavelets}. However,  methods that employ these functions often  require the a priori selection of shape parameters~\cite{rippa1999algorithm}, mother/father wavelets~\cite{debnath2003wavelets}, initial pole configurations~\cite{gustavsen1999rational}, or special rational basis functions~\cite{fridli2011rational}. One must carefully select parameters to avoid numerical instability and computational inefficiency.  In contrast, we introduce flexible, data-driven, general-purpose software tools that can be applied without special knowledge about the locations or types of singularities in the signal. Our methods construct trigonometric rational representations of signals, and we develop a collection of algorithms for computing adaptively and efficiently with them. Our approach combines adaptations of two primary approximation methods,  the AAA algorithm for rational approximation~\cite{nakatsukasa2018aaa} and a Fourier inversion technique based on Prony's method~\cite{beylkin2009nonlinear}. Both of these methods can automatically construct near-optimal rational approximations to functions, and we develop a framework that unifies the representations they construct. The result is an automatic rational approximation method that enjoys two main advantages: (1) it is more robust to various forms of corruption than either method alone, and (2) taken collectively, the two representations are efficient for performing a range of  post-processing operations (see \cref{tab:Operations}).

Periodicity is a common assumption encountered in the practical analysis of signals, and the trigonometric rational functions we use arise naturally from this assumption. However, with some adjustments, our ideas can also be applied using rational functions in the non-periodic setting (see~\cref{sec:nonper}). One might also consider settings where rational functions are fit to samples in frequency space, rather than signal space. Extensive literature on this topic is available from the digital filtering and optimal control communities~\cite{francis1987course, oppenheim2001discrete}.

\subsection{The approximation problem} 
Let $f\colon[0,1)\rightarrow \mathbb{R}$ be an unknown continuous periodic function of bounded variation, and suppose that for some integer $N$, we observe $2N+1$ noisy samples of $f$ at $T \coloneqq \{ x_j\}_{j = 0}^{2N}$, i.e., $y_j = f(x_j) + s_j$ for $0\leq j\leq 2N$. Here, $s_j$ can be: (i) additive white Gaussian noise (i.i.d.~normally distributed), (ii) popcorn noise (sparsely corrupted or arbitrarily large errors), or (iii) bounded deterministic errors. Throughout, we assume $\|f\|_\infty =1$, where $\|\cdot\|_\infty$ is the infinity norm on $[0, 1)$, and that the mean value of $f$ over $[0, 1)$ is $0$. Our central approximation problem is to fit a special class of trigonometric rational functions (see \cref{sec:intro_trig_rat}) to the $2N+1$ samples to construct $r_m$, a type $(m\!-\!1, m)$ trigonometric rational where $m$ is selected adaptively so that the sampling error satisfies $\max_{x_j \in T}|f(x_j) - r_m(x_j)| \leq \epsilon$, where $0 < \epsilon < 1$ is a  tolerance parameter. 

In practice, samples of $f$ are often available under far from ideal circumstances. For example, $T$ may consist of poorly distributed (i.e., not equally-spaced) points due to missing or corrupted data, $N$ may be too small to adequately resolve features of interest, or frequencies of interest may be cut off or otherwise distorted during observations. We are interested in how well $f$ can be recovered under such conditions. 

\subsection{Software} Once an approximant is constructed, we want to compute with it reliably.  Answering the following questions has shaped our software's development: 
\begin{enumerate}[leftmargin=\parindent,align=left,labelwidth=\parindent,labelsep=0pt,topsep=1ex]
\item \textit{Which applications are of importance?} Our methods supply a general-purpose approach for working with noisy samples from periodic, univariate signals. They require no a priori knowledge about locations or types of singularities, and they are designed to be flexible enough for use in a range of applications that involve the detection and identification of events (e.g., ECG and geophone monitoring tasks, engineering and financial applications where change-point detection is important, and in the analysis of dynamical systems, such as periodic contagion modeling).

\item \textit{What properties should our approximants have?}    
Like $f$, we want approximants to be periodic, real-valued, and continuous on $[0, 1)$. We employ trigonometric rational functions, which are the periodic analogue of rational functions. 
\item \textit{What form of approximant should we use?}
Trigonometric rationals can be expressed in many forms, but not all of these forms are numerically stable. 
For stable evaluation and rootfinding, we use the barycentric formula~\cite{berrut1988rational}. For efficient recompression and  operations carried out in Fourier space, we represent the Fourier transforms of our trigonometric rational approximants using short sums of weighted complex exponentials (see \cref{Lemma:Exponential_sum_Fourier_coeffs}). 

\item \textit{Which tools should we provide?} We offer a basic set of computational tools. This includes simple algebraic operations (addition, products), calculus-based operations (integration, differentiation), and tools for filtering, (de)convolving,  and rootfinding/polefinding. Whenever possible, we automatically recompress representations as trigonometric rationals/exponential sums to maintain efficiency. One can combine these tools to perform more complicated tasks. 

\end{enumerate}

Accompanying this work is the open-source code REfit~\cite{wilberREfit}, which is written in MATLAB and uses two classes called \texttt{rfun} and \texttt{efun}.  An \texttt{rfun} object stores a representation of $f$ as a barycentric trigonometric rational function. An  \texttt{efun} object stores a representation of the Fourier transform of $f$ as a weighted sum of complex exponentials.
After an \texttt{rfun} or \texttt{efun} object is constructed, it can be manipulated and analyzed through the operations implemented in the package (see~\cref{tab:fun}). The commands are overloaded so that they can be applied to either type of object, and binary operators can be used between objects of different type. 
\begin{table}[htbp] 
\centering
\caption{A selection of REfit commands.}
\begin{tabular}{cc}
 command & Operation \\
\hline \\[-8pt]
\texttt{+}, \texttt{-}, \texttt{.*}, \texttt{./} & basic arithmetic \\ 
\texttt{diff($\cdot$)}, \texttt{cumsum($\cdot$)} & differentiation, indefinite integration \\
\texttt{conv($\cdot$)} & convolution \\
\texttt{corr($\cdot$,$\cdot$)} & cross-correlation
\end{tabular} 
\label{tab:fun} 
\end{table} 

\vspace{.1cm}

\subsubsection{Connections to other work} In addition to the AAA algorithm~\cite{nakatsukasa2018aaa}  and a large body of work on Prony's method~\cite{beylkin2005approximation, beylkin2009nonlinear, plonka2016application, potts2013parameter, prony1795essai}, several recent developments in rational computing are related to our work. In~\cite{derevianko2021exact}, the authors also develop approximation methods that combine Prony's method with the AAA algorithm. Their work focuses on the recovery of parameters for extended models of exponential sums with polynomial coefficients\footnote{This corresponds to using rational functions with poles of higher multiplicity. We find that it is sufficient for our purposes to use clusters of simple poles.} that can also include purely oscillatory terms, whereas our focus is on the development of general-purpose software and rational approximation tools. We consider the setting where exponential sums are applied in Fourier space (corresponding to trigonometric rationals in signal space), and in contrast, they use barycentric rational approximations in Fourier space to recover sums of exponentials in signal space. For this reason, they do not make use of trigonometric rational functions. While the models proposed here for the recovery of signals are different from those in~\cite{derevianko2021exact}, the developments in~\cite{derevianko2021exact} are relevant to our setting.  For example, the derivatives of a trigonometric rational function correspond to exponential sums in Fourier space with polynomial coefficients. The parameters of such a model can be exactly recovered using the ideas in~\cite{derevianko2021exact}. 

In~\cite{baddoo2020aaatrig}, a trigonometric variant of AAA is introduced in the context of conformal mapping.  It is similar to the pronyAAA method we introduce in \cref{sec:trigAAA}.  However, a key difference is that it is not modified to account for any special connection to approximations in Fourier space.  Other popular and effective rational approximation methods include the RKFIT algorithm (and software package)~\cite{berljafa2017rkfit},  the vector fitting method~\cite{gustavsen1999rational}, and the Loewner framework of Mayo and Antoulas~\cite{karachalios2021loewner}, which is closely connected to the AAA algorithm. These methods have not been explicitly adapted to the periodic setting, and the first two methods are not suited to our needs as they require and can be sensitive to a set of initialization parameters (e.g., guesses of the location of the poles). Recently, AAA has been combined with RKFIT to overcome this difficulty~\cite{elsworth2019conversions}.  The development of our software package is largely inspired by Chebfun~\cite{Chebfun, wright2015extension}, which is primarily designed for computing automatically and adaptively with smooth functions. In some respects, our software serves as an analogue to Chebfun for computing with non-smooth functions.

The rest of this paper is organized as follows: In \cref{sec:intro_trig_rat} we briefly review trigonometric rational functions and the barycentric form. We introduce the trigonometric variant of the AAA algorithm that we use to construct barycentric trigonometric interpolants (see \cref{sec:pronyAAA}). In \cref{sec:Pronyappx}, we apply the regularized Prony's method (RPM) to construct approximations in Fourier space. In \cref{sec:transforms}, we introduce stable Fourier and inverse Fourier transform methods for moving between representations in the time and frequency domains. 
In~\cref{sec:Examples,sec:REfit}, we give 
examples and descriptions of the algorithms for computing with these representations.

\section{Trigonometric rational functions and their Fourier transforms}
\label{sec:intro_trig_rat}
The trigonometric rationals are the periodic analogue of rational functions~\cite{henrici1979barycentric, javed2016algorithms}. A trigonometric rational of period $1$ is the quotient of two trigonometric polynomials of period $1$, i.e., a function of the form
\begin{equation}
\label{eq:Rationalpqform}
r(x) = \dfrac{p_\ell(x)}{q_m(x)} = \frac{\sum_{j = -\ell}^\ell a_j e^{2 \pi i jx}}{\sum_{j = -m}^m b_j e^{2 \pi i j x}}, \quad x \in [0, 1).
\end{equation}
We call $r$ a type $( \ell, m)$ trigonometric rational function. We follow the convention that unless stated explicitly, a type $(\ell, m)$ function is in reduced form, meaning that $p_\ell$ and $q_m$ have no zeros in common.  We restrict our interest to the family of period $1$ trigonometric rationals $r_m$ of type $(m\!-\!1, m)$ that are real-valued and continuous on $[0, 1)$. Furthermore, we assume that the roots of the denominator $q_m$ are simple. Under these assumptions, if $\eta_j$ is a root of $q_m$, then so is its complex conjugate $\overline{\eta}_j$, as well as $\eta_j \pm K$, where  $K$ is any integer. We say that a pole of $r_m$ is any root, $\eta_j$, of $q_m$ such that $0 \leq Re(\eta_j)<1$.  In the same way, any root of $p_{m-1}$ with a real part in the interval $[0, 1)$ is called a zero of $r_m$.  

\subsection{Why trigonometric rationals?} Three key properties of this family of functions make them ideal for our setting. First, like standard rational functions, they are especially effective at resolving singularities.  For example, for particular choices of $f$, such as \hbox{$f(x) = |x - 1/2|-1/4$}, it is known that trigonometric rationals of type $( m, m)$ can converge to $f$ at a root-exponential rate with respect to $m$~\cite{trefethen2013approximation}.
\begin{figure}
\hspace{.2cm}
    \begin{minipage}{.48\textwidth} 
 \centering
  \begin{overpic}[width=.95\textwidth]{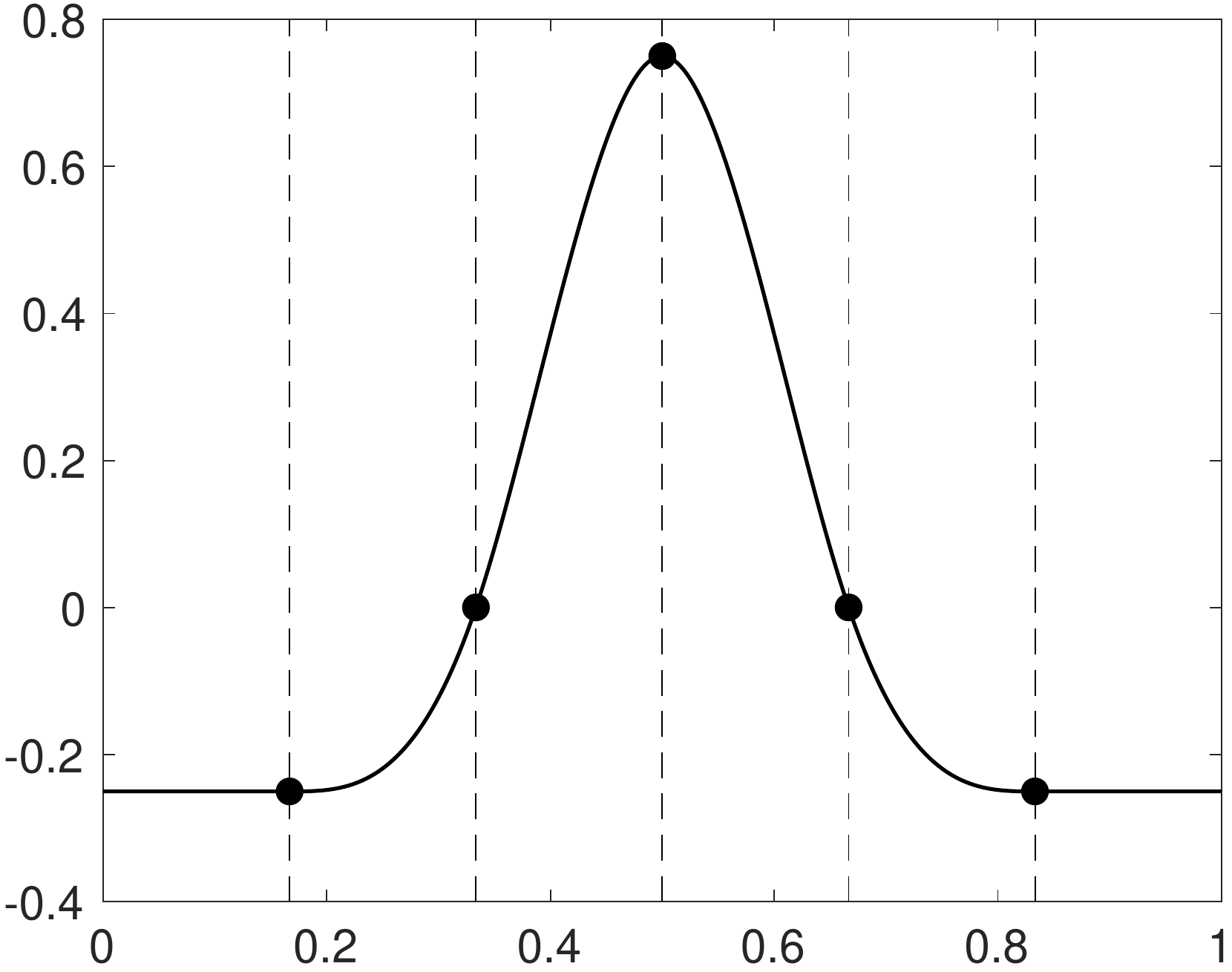}
    \put(50,-5){$x$}
  \end{overpic}
  \end{minipage}
  \hspace{.2cm}
      \begin{minipage}{.46\textwidth} 
 \centering
  \begin{overpic}[width=.78\textwidth]{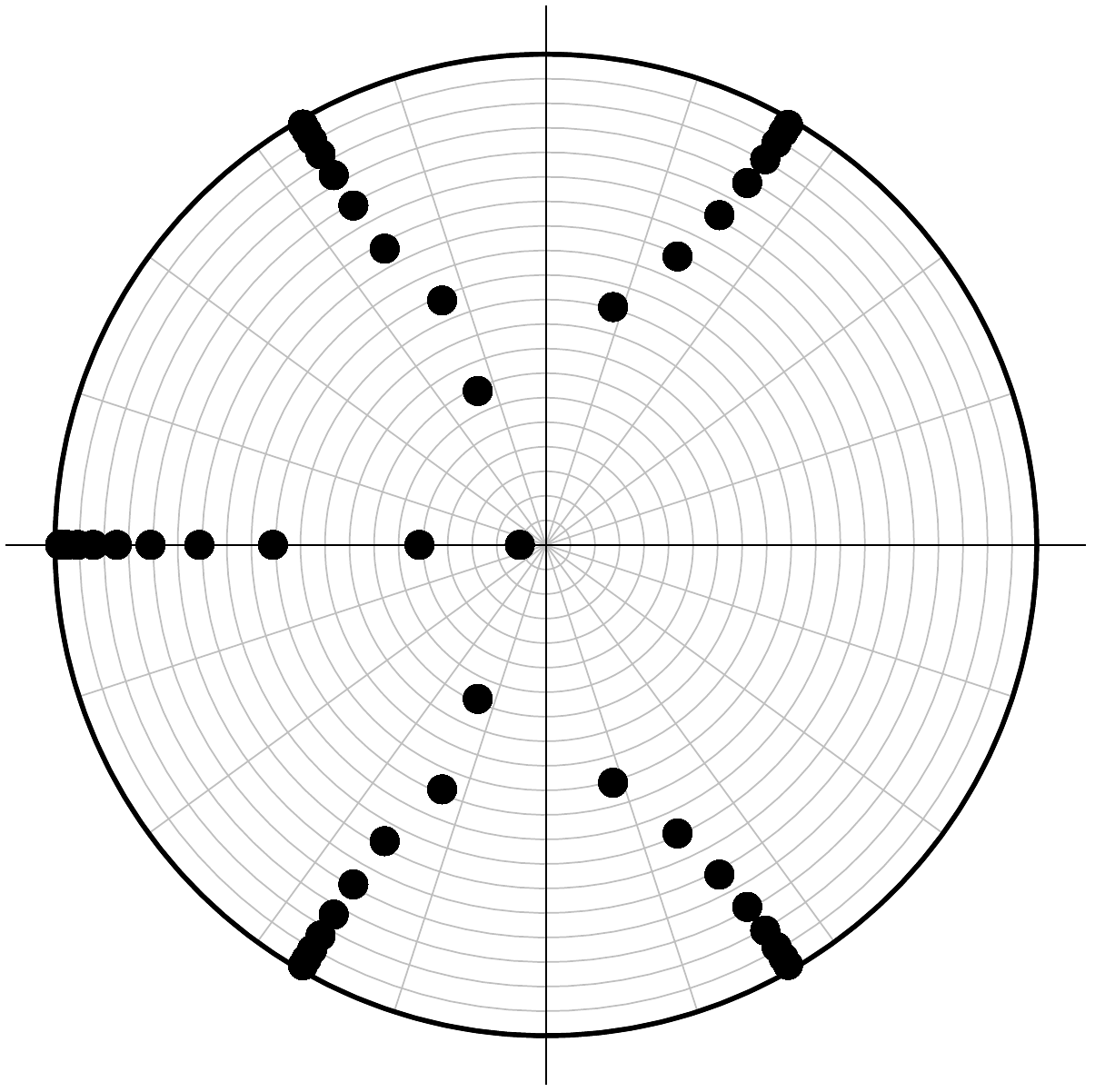}
    \put(102, 49){$Re$}
    \put(46, 102){$Im$}
  \end{overpic}
  \end{minipage}
  \caption{ Left: A trignometric rational function \smash{$r_m$} ($m = 44$) approximating the shifted and scaled cubic B-spline on the interval $[0, 1)$ is constructed and plotted. The knots of the spline occur at the dotted vertical lines. Away from the singularities, the absolute approximation error is on the order of \smash{$10^{-11}$}.  Approximate knot locations (black dots) are automatically computed using the poles of the rational \smash{$\tilde{r}_m(z) = r_m(x)$}, where \smash{$z = e^{2\pi i x}$} (see~\cref{eq:zt_form}).   Right: The poles of \smash{$\tilde{r}_m$} of magnitude $\leq 1$ are plotted in the unit disk in the complex plane. They cluster toward the points \smash{$e^{2 \pi i x_k}$} on the unit circle, where each \smash{$x_k$} is a knot in the spline. } 
  \label{fig:pole_plot}
  \end{figure}
Second, these functions can be represented efficiently in Fourier space. We make extensive use of the following  fact: 
\begin{lemma}
\label{Lemma:Exponential_sum_Fourier_coeffs}
Let  $r_m(x)$ be a type $(m\!-\!1, m)$ nonzero trigonometric rational function that is real-valued and continuous  on $[0, 1)$ with exactly $2m$ simple poles.  Let $\{ \eta_j\}_{j = 1}^m$ denote the collection of poles with $Im(\eta_j) > 0$. If the Fourier coefficients of $r_m$ are given by $\{(\hat{r}_m)_k\}_{k= -\infty}^{\infty}$, then there exist $\omega_j$ such that 
\begin{equation}
\label{eq:Exp_Sum}
(\hat{r}_m)_k  = R_m(k): = \begin{cases} 
\sum_{j=1}^m \omega_j e^{ \alpha_j k},  & k \geq 0, \\
\sum_{j=1}^m \overline{\omega}_j e^{- \overline{\alpha}_j k},  & k < 0,
\end{cases}
\end{equation}
where  $\alpha_j = 2\pi i \eta_j$. Here, $\overline{\omega}_j$ is the complex conjugate of $\omega_j$. 
\end{lemma}
\begin{proof}  See~\cite[Ch.~4]{stein2010complex}.  
\end{proof}
\Cref{Lemma:Exponential_sum_Fourier_coeffs} shows that the Fourier series of $r_m$  can be represented as sum of $m$  weighted decaying exponentials. We say that $\mathcal{F}(r_m) = R_m$ is the Fourier transform of $r_m$, and similarly, $\mathcal{F}^{-1}(R_m) = r_m$ is the inverse Fourier transform of $R_m$.  

Third, trigonometric rationals can be used for feature detection.  For example, in \cref{fig:pole_plot}, we plot the poles $z_j = e^{2 \pi i \eta_j}$, $Im(\eta_j) > 0$, of the rational function 
\begin{equation}
\label{eq:zt_form}
\tilde{r}_m(z) = \dfrac{z^m p_{m-1}(z)}{z^m q_{m}(z) }, \quad z = e^{2 \pi i x}, 
\end{equation}
where $ r_m(x) = p_{m-1}(x)/ q_m(x)$ is a trigonometric rational approximation to a shifted and scaled uniform cubic B-spline~\cite{de1972calculating} $f$ with knots at  \hbox{$x =\{1/6, \ldots, 5/6\}$}. The knots are not easily identifiable in a plot of $f$, but the poles of $\tilde{r}_m$, which are chosen adaptively via \cref{alg:apm}, cluster toward the singularities and reveal their locations. Using the five poles in $\{ z_j\}_{j = 1}^{m}$ with the largest magnitudes, estimates of the knot locations are given by 
$ \tilde{x}_k =  {\rm arg} (z_k)/ 2 \pi $ for $1 \leq k \leq 5$.   
 The radial coordinates of the poles in \cref{fig:pole_plot} also encode information about $f$ in the frequency domain. The Fourier coefficient  $\hat{f}_k$ is well-approximated by $(\hat{r}_m)_k = R_m(k)$, where $R_m$ is as in~\cref{eq:Exp_Sum}. The terms in $R_m$ with $|z_j| \ll 1$ have negligible influence when $k$ is small, so they capture aspects of the signal that are only observable at low frequencies. 

\subsection{Barycentric trigonometric rational functions} \label{sec:bary}
Large numerical errors can be incurred when evaluating trigonometric rationals that are numerically constructed using~\cref{eq:Rationalpqform} directly~\cite{filip2018rational, trefethen2013approximation}. Another natural way to represent $r_{m}$ is in a pole-residue format with respect to the pairs of poles $(z_j = e^{2 \pi i \eta_j}, 1/\overline{z}_j)$ of the rational function $\tilde{r}_m$ in~\cref{eq:zt_form}.  Evaluation  using this format is often reliable in practice, but it can potentially result in catastrophic cancellation if the evaluation point is too close to the poles. 
In the non-periodic setting, the AAA algorithm~\cite{nakatsukasa2018aaa}  safeguards against such instabilities by using barycentric rational interpolants with backward-stable evaluation on the interval of approximation~\cite{austin2017numerical, higham2004numerical}. Type $(m\!-\!1,m\!-\!1)$ barycentric rational interpolants of a function $f$ on $[0, 1)$ have the form
\begin{equation} 
\label{eq:std_bary}
v_m^{\gamma, t} = \dfrac{\tilde{n}_{m}(x)}{\tilde{d}_m(x)} =  \sum_{j = 1}^{m} \dfrac{\gamma_j }{x-t_j}f(t_j) \Bigg/ \displaystyle{\sum_{j = 1}^m} \dfrac{\gamma_j}{x-t_j} , 
\end{equation}
where for  $1\leq j \leq m$, $t_j \in [0, 1)$. The trigonometric analogue of these interpolants is described in~\cite{henrici1979barycentric}, and more recently, in~\cite{baddoo2020aaatrig, javed2016algorithms}.  Type $(m\!-\!1, m)$ barycentric trigonometric rationals on $[0, 1)$ take the following form: 
 \begin{equation}
 \label{eq:BarycentricRat}
 r_m^{\gamma, t}(x) = \dfrac{n_{m-1}(x)}{d_m(x)} = \dfrac{ \sum_{j = 1}^{2m} \gamma_j f_j \cot\left(\pi (x-t_j)\right)}{ \sum_{j = 1}^{2m} \gamma_j \cot\left(\pi (x-t_j)\right)}, \qquad \sum_{j = 1}^{2m} \gamma_jf_j = 0.
 \end{equation}
The interpolating points $t = \{t_1, \ldots, t_{2m}\}$, always assumed to be  distinct, are called barycentric nodes,  $\gamma = \{\gamma_1, \ldots, \gamma_{2m}\}$ are called barycentric weights, and $f_j = f(t_j)$.
 It is easily shown that $r_m^{\gamma, t}(t_j) = f_j$ whenever each $\gamma_j$ is nonzero~\cite{berrut1988rational}, but it is not obvious from~\cref{eq:BarycentricRat} that $r_m^{\gamma, t}$ is in fact a type $(m-1, m)$ trigonometric rational. This can be seen by using the trigonometric polynomial \smash{$\ell_t(x) = \prod_{j = 1}^{2m}\sin(\pi(x -t_j))$}, and rewriting $r_m^{\gamma, t}$ as \smash{$ \ell_t n_{m-1}  / \ell_t d_m $} (see~\cite{henrici1979barycentric}).  The condition $\sum_{j = 1}^{2m} \gamma_j f_j = 0$ enforces that the numerator is a trigonometric polynomial of degree $m\!-\!1$, rather than $m$. It is not clear from~\cref{eq:BarycentricRat} where the poles of $r_m^{\gamma, t}$ lie, and in particular, whether they lie off $[0, 1)$ (see \cref{sec:pronyAAA}).  Stability properties (and other advantages) of the barycentric form were popularized in the context of polynomial interpolation~\cite{berrut2004barycentric}, where  $\gamma$ is always chosen so that $d_m (x)= 1$. Analyses can be found in~\cite{austin2017numerical, higham2004numerical}, with discussion on stability properties in the more general case, where $d_m(x) \neq 1$,  in~\cite[Sec.~2.2--2.3]{filip2018rational}.  

\subsection{Approximations in time} \label{sec:pronyAAA}
 In the noiseless setting, the barycentric rational $r_m^{\gamma, t}$ can be directly constructed from samples of $f$ using an analogue of the AAA algorithm (pronyAAA) that we introduce here.  This supplies a fast, automated way to construct trigonometric rational representations of signals.  
 A major advantage of AAA-type methods over other approximation schemes is that they can be blithely applied to samples from non-uniform grids and can even be used for recovering functions defined on disjoint sets of support~\cite{nakatsukasa2018aaa}.  In particular, these methods are robust to missing samples. In \cref{fig:missing_data}, we use pronyAAA to recover a function from data that has been deleted in several corrupted regions. 

\begin{figure}
\centering
\hspace{.2cm}
    \begin{minipage}{.45\textwidth} 
 \centering
  \begin{overpic}[width=\textwidth]{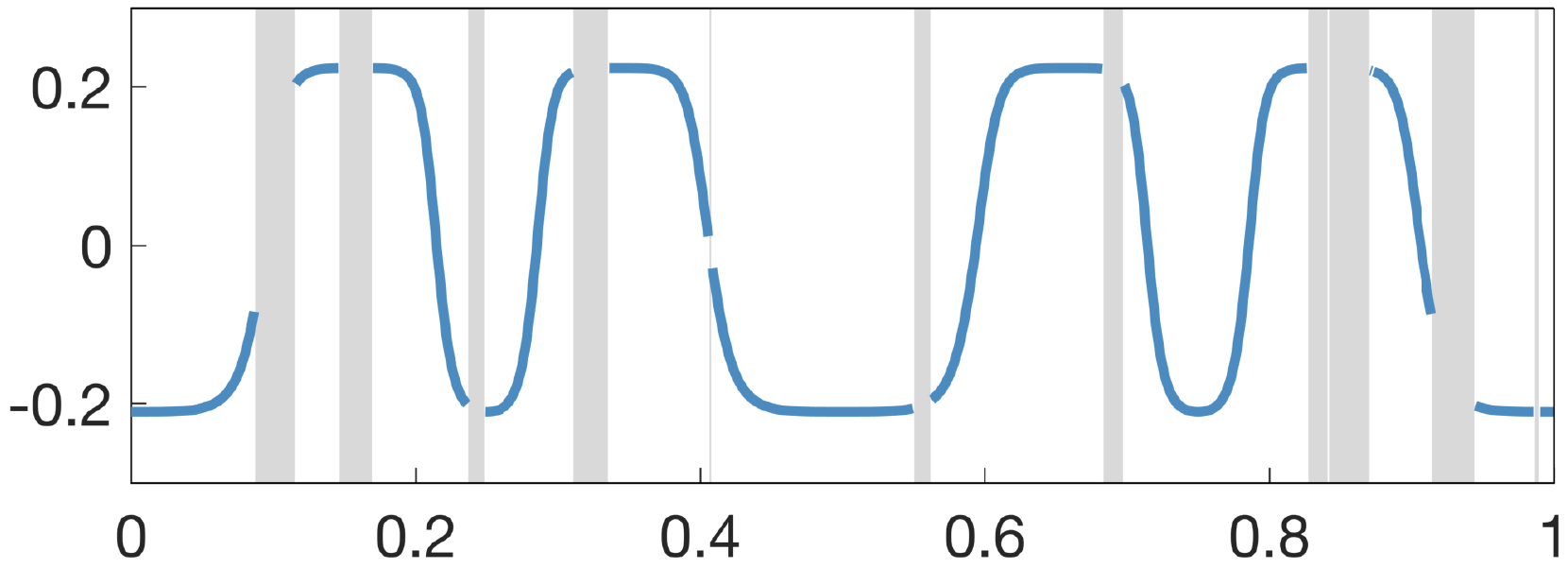}
  \end{overpic}
    \begin{overpic}[width=\textwidth]{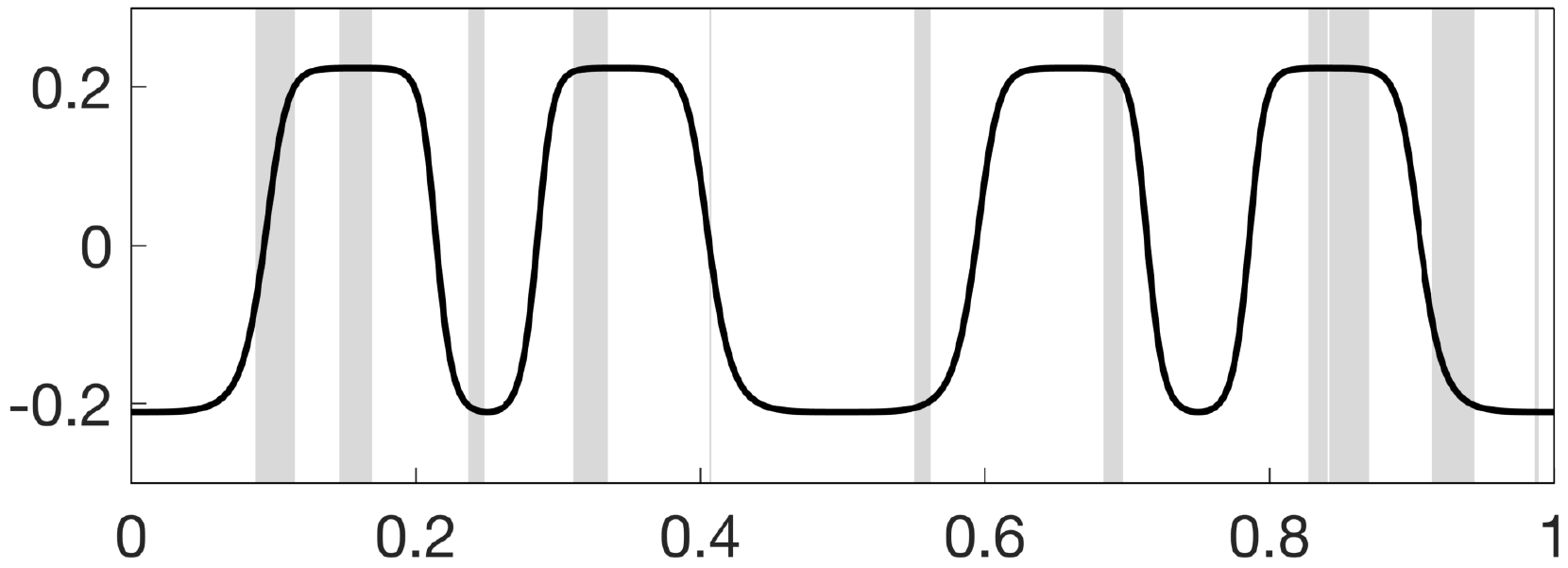}
   \put(-7, 22){\rotatebox{90}{amplitude}}
    \put(50, -5){\rotatebox{0}{$x$}}
  \end{overpic} 
  \end{minipage}
  \hspace{.2cm}
      \begin{minipage}{.45\textwidth} 
 \centering
  \begin{overpic}[width=.9\textwidth]{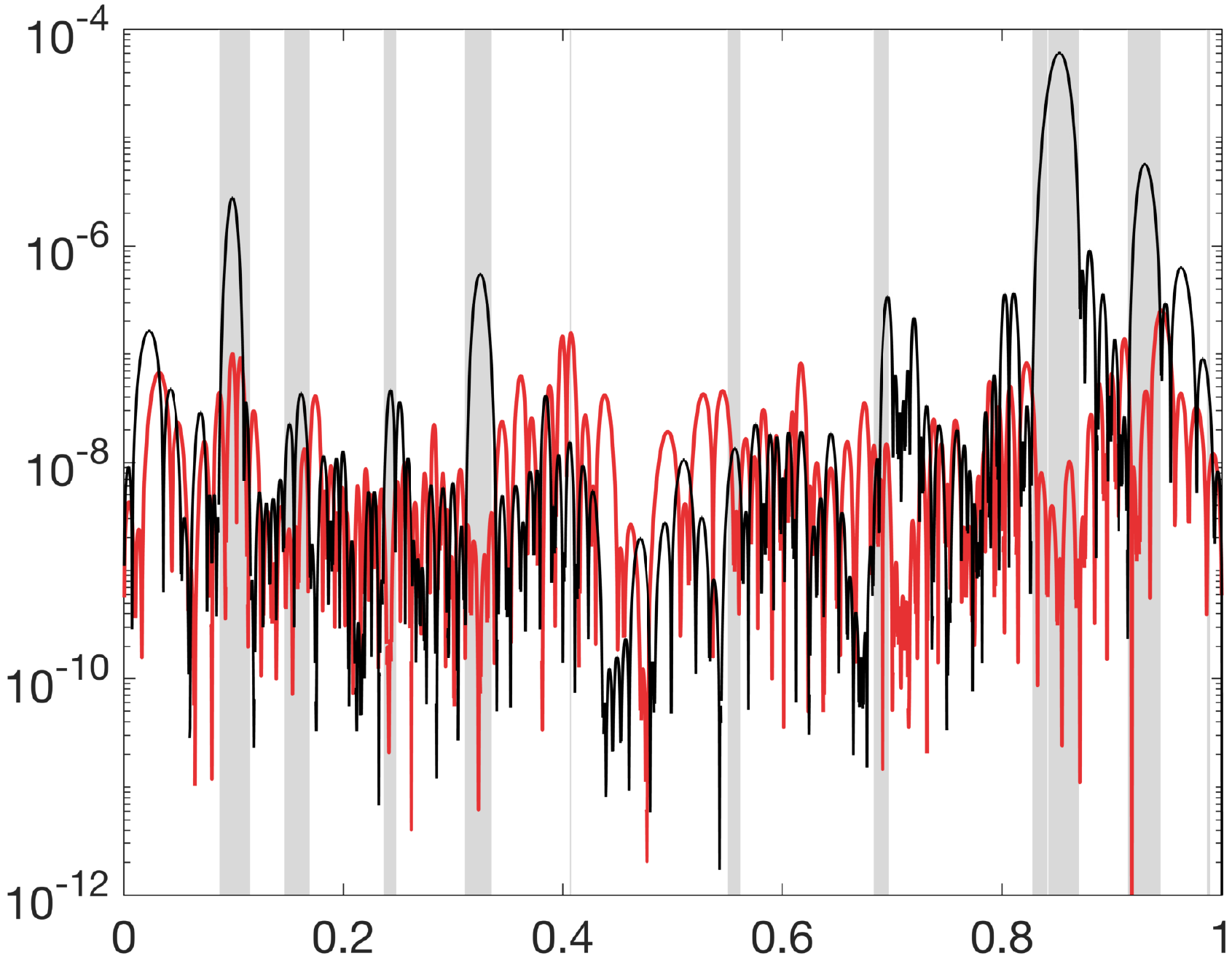}
  \put(-6, 35){\rotatebox{90}{error}}
  \put(50, -5){\rotatebox{0}{$x$}}
  \end{overpic}
  \end{minipage}
  \caption{ Left: $3000$ samples from a function $f$  are taken on an equally spaced grid from $[0, 1)$. However, the data inside the grey regions are corrupted and thus deleted from the sample (a total of $609$ observations are missing). The available data are plotted (blue, upper panel) and used to construct a type $(34, 35)$ trigonometric barycentric rational interpolant $r_{35}$ via pronyAAA. The resulting approximant, evaluated on an equally spaced grid and plotted (black, lower panel), imputes the missing data in the grey regions. Right: The absolute error $|f(x) - r_{35}(x)|$  is plotted in black on a logarithmic scale against the original grid of sampled points. The absolute error when pronyAAA is applied to an uncorrupted sample of $f$ with no missing data is also displayed (red). } 
  \label{fig:missing_data}
  \end{figure}

 \subsubsection{Constructing barycentric trigonometric interpolants}
\label{sec:trigAAA} 
 Our construction algorithm is a straightforward extension of the AAA algorithm (see \cref{alg:AAA}), with differences being the use of trigonometric basis functions, the restriction that the numerator is always of degree $m\!-\!1$, and the requirement that the constructed interpolant has an even number of interpolating points.\footnote{This is so that  $r_m^{\gamma, t}$ may have an even number of simple poles.} These restrictions are so that the approximant corresponds to a length $m$ exponential sum in Fourier space. As with AAA, pronyAAA employs a greedy residual minimization method to select interpolating points and iteratively build up an interpolant. We briefly discuss the process here, and refer to~\cite{nakatsukasa2018aaa} for details.

Let $\{f_0, \ldots, f_{2N}\}$ be the samples of a function $f$ we seek to approximate,  where
 $T = \{x_j\}_{j = 0}^{2N} \subset [0, 1)$ are the sample locations and $f(x_j) = f_j$. We first describe how  the barycentric weights are updated at each iteration. Suppose that at the $m^{\text{th}}$ iteration, the barycentric support points for $r_{m}^{\gamma, t}(x)$ in~\cref{eq:BarycentricRat} are \hbox{$t = \{t_1, \ldots, t_{2m}\} \subset T$}. Let $\widetilde{T} = T \setminus \{ t_1, \ldots, t_{2m}\}$. We must now choose $\{ \gamma_j\}_{j = 1}^{2m}$.   Let \hbox{$\gamma = (\gamma_1, \ldots, \gamma_{2m})^T$}. As in~\cref{eq:BarycentricRat}, $r_{m}^{\gamma, t} = n_{m-1}/ d_m$, where $n_{m-1}$, $d_m$ are trigonometric polynomials.  We select  $\gamma$ as the solution to the following constrained optimization problem: 
 \begin{equation}
 \label{eq:AAAopt}
 \min_{\gamma \in \mathbb{C}} \sum_{x_j \in \widetilde{T}} \left( f(x_j) d_{m} (x_j)-n_{m-1}(x_j)  \right)^2, \quad \textnormal{s.t.  } \sum_{j = 1}^{2m} f(t_j)\gamma_j = 0, \quad \|\gamma\|_2 = 1.
 \end{equation} 
This  is similar to the constraint applied in the standard AAA algorithm (see step 2, line (iii), of \cref{alg:AAA}), except that it additionally requires  $n_{m-1}$ to be of  degree $m\!-\!1$.
The constraint is satisfied if we select a vector of the form $\gamma = Q \tilde{\gamma}$, $\|\tilde{\gamma}\|_2 = 1$, where  $Q$ is a $2m  \times  (2m\!-\!1)$ matrix with orthonormal columns that span the space orthogonal to the vector $(f(t_1), \ldots, f(t_{2m}))$.  
Since for each $x_j \in \widetilde{T}$, 
\begin{equation} 
\label{eq:linearizedOpt}
f(x_j) d_{m} (x_j)-n_{m-1}(x_j)  = \sum_{\ell = 1}^{2m} \left( f(x_j) -f(t_\ell) \right)  \gamma_\ell \cot(\pi x_j -  \pi t_\ell), 
\end{equation}
 we see that $\tilde{\gamma}$ is given by the last right singular vector of the matrix $CQ$, where $C \in \mathbb{R}^{|\widetilde{T}| \times 2m}$ has entries of the form \smash{$C_{j \ell} = (f(x_j) - f(t_\ell)) \cot (\pi x_j - \pi t_\ell)$}, where each \smash{$x_j$} is a unique member of \smash{$\widetilde{T}$}. 

Once $\gamma$ is computed, we have constructed $r_{m}^{\gamma, t}$ and the iteration is complete.  If the error \smash{$ \max_{x_j \in \widetilde{T}} |f(x_j)-r_{m}^{\gamma, t}(x_j) |$} is not sufficiently small, we begin the next iteration by choosing two additional interpolating points (as we require an even number of interpolating points) from $\widetilde{T}$. 
It is not clear how to efficiently choose two points in a single step, so we introduce an interim step. The first point is chosen as \smash{$t_{2m+1} = {\rm argmax}_{x_j \in \widetilde{T}} |f(x_j) - r_{m}^{\gamma, t}(x_j)|$}, and then $t$ and $\widetilde{T}$ are updated appropriately. To pick the second point, we construct an interim trigonometric interpolant \smash{$r_{m+1/2}^{\gamma, t}$} with an odd number of interpolating points, using the basis functions \smash{$\csc(\pi (x - t_j))$} instead of \smash{$\cot(\pi (x - t_j)).$}  The different basis functions ensure that \smash{$r^{\gamma, t}_{m+1/2}$} is a trigonometric rational: see~\cite{henrici1979barycentric}. The weights in \smash{$r_{m+1/2}^{\gamma, t}$} are updated by solving a problem similar to~\cref{eq:AAAopt}, though the constraint differs due to the different basis functions~\cite{henrici1979barycentric}. The second interpolating point is then chosen as \smash{$t_{2m+2} = {\rm argmax}_{{x_j}\in \widetilde{T}} |f(x_j) - r_{m+1/2}^{\gamma, t}(x_j)|$}. A different strategy for choosing the second point at each step is used in~\cite{baddoo2020aaatrig}, but we find that such a choice adversely impacts convergence in our setting.\footnote{One might wonder if it is possible to choose the interpolation points so that the pole locations are restricted (or symmetries in pole locations are enforced). Such a challenge is of interest in rational-based numerical methods~\cite{costa2021aaa, xu2013bootstrap}. The Floater--Hormann interpolants are one such example~\cite{berrut2005recent}, but they have slower convergence rates than AAA-based interpolants.}

 \begin{algorithm}[t!]
\caption{The standard AAA algorithm.}
\label{alg:AAA}
\begin{algorithmic}
\STATE \hspace{-.45cm} 
\STATE \textbf{Input:}  tolerance parameter $\epsilon$, sample locations $T = \{ x_0, \ldots, x_{2N}\}$,  samples $\{f(x_0),  \ldots, f(x_{2N}) \}$. 
\STATE \textbf{Output:} barycentric support points $t$ and weights $\gamma$ defining $v_m^{\gamma, t}$ in~\eqref{eq:std_bary}. 
\vspace{.2cm} 
\STATE 1. Set $err = 1$, $m = 0$, $t= \{ \}$, $\widetilde{T} = T$, $v_0^{\gamma, t} = 0$. 
\vspace{.2cm} 
\STATE 2. while $err > \epsilon$
\STATE \hspace{1cm} (i)  Set \smash{$t_{m+1}= {\rm argmax}_{x_j \in \widetilde{T} } \{ |f(x_j)-v_m(x_j)|\}$}. 
\STATE \hspace{1cm} (ii)  $t \leftarrow t \cup \{t_{m+1}\}$, $\widetilde{T} \leftarrow T\setminus t$. 
\STATE \hspace{1cm}  (iii) With \smash{$\tilde{n}_{m+1}$}, \smash{$\tilde{d}_{m+1}$} as in~\eqref{eq:std_bary}, choose $\gamma = (\gamma_1, \ldots, \gamma_{m + 1})^T$ such that it 
\STATE \hspace{1cm} solves the following constrained optimization problem: 
\STATE \hspace{2cm} minimize $ \sum_{j = 1}^{|\widetilde{T}|} \left( f(x_j) \tilde{n}_{m+1}(x_j) - \tilde{d}_{m+1}(x_j) \right)^2$  s.t. $\|\gamma\|_2 = 1$.
\STATE   \hspace{1cm} (iv) Set $err \leftarrow \max_{ x_j \in \widetilde{T} }|f(x_j) - v_{m+1}^{\gamma, t}(x_j)|$.
\STATE   \hspace{1cm} (v) $m \leftarrow m+1$.
\STATE \phantom{2.} end while
\STATE 4. Compute poles and residues (see \cref{sec:Roots} and~\cite[eq.~3.11]{nakatsukasa2018aaa}). 
\STATE 5. If there are spurious poles, apply cleanup (see \cref{cleanup}  and~\cite[sec.~4]{nakatsukasa2018aaa}). 

\end{algorithmic}
\end{algorithm}

\subsubsection{Spurious poles} \label{cleanup} 
As with AAA, nothing in the pronyAAA algorithm directly controls where the poles of $r_m^{\gamma, t}$ occur. The advantage of this approach is that, unlike fixed-pole approximation methods, the pole locations are adaptively determined and automatically cluster in patterns that are highly effective for resolving singularities~\cite{trefethen2021exponential}. However, one drawback to this approach is that the conjugate-pair symmetry of the poles is not explicitly enforced. A more problematic issue is that so-called ``spurious poles" may appear on the interval of approximation. 

Poles that are undesirable or unnecessary are called spurious~\cite{berrut1988rational, nakatsukasa2018aaa}.  They arise as artifacts related to numerical error, but can also appear within mathematically valid solutions to an interpolation problem constrained by~\cref{eq:AAAopt}. For example, a spurious pole may co-occur with a nearby zero that effectively cancels out its influence except within a small interval $I$, where $I \subset [x_j, x_{j+1}]$ for some neighboring sample locations $x_j < x_{j+1}$, so that for $x \notin I$ it holds that $|r_m^{\gamma, t}(x) - f(x)| < \epsilon$. This is a perfectly acceptable way to solve~\cref{eq:AAAopt}, but  it leads to a solution where $\|r_m^{\gamma, t} - f\|_\infty $ is unbounded.  

 \begin{algorithm}[t!]
\caption{The pronyAAA algorithm.}
\label{alg:pronyAAA}
\begin{algorithmic}
\STATE \hspace{-.45cm} 
\STATE \textbf{Input:}  tolerance parameter $\epsilon$, sample locations $T = \{ x_0, \ldots, x_{2N}\}$,  samples $\{f(x_0),  \ldots, f(x_{2N}) \}$. 
\STATE \textbf{Output:} barycentric support points $t$ and weights $\gamma$ defining  $r_m^{\gamma, t}$ in~\cref{eq:BarycentricRat}. 

\vspace{.2cm} 

\STATE 1. Set $m = 1$,  $t = \{ t_1 = {\rm argmax }_{x_j \in T}|f(x_j)|, t_2 =  {\rm argmax}_{x_j \in T\setminus t_1}|f(x_j)|\}$, $\widetilde{T} = T \setminus t$. \\
2. Solve for $\gamma$ as in \cref{eq:linearizedOpt} to define $r_1^{\gamma, t}$. Set $err = \max_{x_j \in \widetilde{T}} |f(x_j) - r_1(x_j)|$.

\vspace{.2cm}

\STATE 3. While $err > \epsilon$
\STATE \hspace{1cm} for $\ell = 1, 2$ 
\STATE \hspace{2cm}  (i) Set $t_{2m+\ell} =  {\rm argmax}_{x_j \in \tilde{T}}|f(x_j) - r_{m+ (\ell-1)/2} ^{\gamma, t}| $. 
\STATE \hspace{2cm}  (ii) $t \leftarrow \{ t_1, \ldots, t_{2m+\ell}\}$,   \hspace{.1cm}  $\widetilde{T} \leftarrow \widetilde{T} \setminus t_{2m+\ell}$.
\STATE \hspace{2cm}  (iii) Update $\gamma = (\gamma_1, \ldots, \gamma_{2m+\ell})^T$ by solving a constrained \\ \hspace{2cm} optimization problem on $\widetilde{T}$ (see \cref{eq:linearizedOpt} and discussion). 
\STATE \hspace{1cm} end for
\STATE   \hspace{1cm} $err \leftarrow \max_{x_j \in \widetilde{T}}|f(x_j) - r_{m+1}^{\gamma, t}(x_j)|$.
\STATE   \hspace{1cm} $m \leftarrow m+1$.
\STATE \phantom{2.} end while

\vspace{.2cm} 

\STATE 4. Compute poles and residues (see \cref{sec:Roots}). 

\STATE 5. If there are spurious poles, apply cleanup routine (see \cref{cleanup}). 
\end{algorithmic}
\end{algorithm}

 Spurious poles are eliminated in the standard AAA algorithm with an additional cleanup routine~\cite{nakatsukasa2018aaa}, which we adapt to our setting. Pairs of spurious poles are detected via their small residues or by the fact that they are real-valued. The barycentric nodes nearest to the spurious poles are eliminated. This forces the degree of the interpolant to decrease and requires the barycentric weights to be recomputed, which changes the number and location of the poles.  When poles cannot be eliminated without destroying the approximation, it is usually because the function being interpolated is not well-approximated by low to moderate type $(m\!-\!1, m)$ trigonometric rationals. We observe this, for example, when pronyAAA is applied to samples perturbed by additive Gaussian noise. An example is described in \cref{sec:ECG} that involves the reconstruction of ECG data from $645$ samples. The direct application of  pronyAAA to the noisy data is problematic. It results in $62$ spurious poles appearing on $[0, 1)$ that create artificial peaks in the signal reconstruction. To overcome this, we need a method of rational approximation that is robust to Gaussian noise. 
 
 \subsection{Approximations in Fourier space} \label{sec:Pronyappx}
The pronyAAA algorithm constructs trigonometric rational representations to signals directly from samples, but its exclusive use is inadequate in many practical settings. For example, one cannot apply it in the presence of Gaussian noise, and the barycentric form is not conducive to efficient recompression techniques (see \cref{sec:comp}). These issues can often be remedied by representing $f$ in Fourier space using the exponential sums in~\cref{eq:Exp_Sum}. 

To construct the sums, we require the Fourier coefficients of $f$. We use the fast Fourier transform (FFT) to compute the coefficients $v = (\hat{f}_0 , \ldots, \hat{f}_{N})^T$ associated with samples $\{f(x_j)\}_{j = 0}^{2N}$, where $x_j = j/(2 N+1)$. When $f$ is not a bandlimited function (or has a bandlimit $> N$), this process introduces error into the Fourier coefficients. We describe the error with the following notion:
\begin{definition}[$\epsilon$-resolution]
For  $0 < \epsilon < 1$, the $\epsilon$-resolution of $f$ is the smallest non-negative integer $N_\epsilon$ such that  
$$ \| f - f_{trunc}\|_\infty \leq \epsilon,$$
where $f_{trunc}$ is the best $\mathcal{L}^{\infty}([0, 1))$ projection of $f$ onto the functions of bandlimit $N_\epsilon$. 
\end{definition}
For bandlimited functions, $N_0$ is the bandlimit of $f$. The $\epsilon$-resolution for non-bandlimited functions can be understood in relation to the smoothness of $f$ and its region of analyticity in the complex plane~\cite{trefethen2013approximation}. In our setting, the assumption is that $N_\epsilon$ is large. However, when $f$ is well-approximated by a type $(m\!-\!1, m)$ trigonometric rational, \cref{Lemma:Exponential_sum_Fourier_coeffs} indicates that $\mathcal{F}(f)$ can be represented with far fewer degrees of freedom via the exponential sum $R_m$. One way to find $R_m$ is by fitting the nonlinear model \smash{$R_m(k) =\sum_{j = 1}^{m} \omega_j e^{\alpha_j k}$} to the Fourier coefficients $\hat{f}_k$, $0 \leq k \leq N$. We emphasize that  $m$ is unknown in the general setting and must be determined adaptively. 

\subsubsection{Regularized Prony's method} To construct $R_m$, we follow an idea in~\cite{beylkin2009nonlinear} and use the regularized version of Prony's method (RPM) from~\cite{beylkin2005approximation}. Variants of this method go by many names across various disciplines, and we refer to~\cite{potts2013parameter} for an overview. 
The problem of finding $R_m$ can be recast as a structured low rank approximation problem involving Hankel matrices. One can understand the connection using the following lemma, a version of which was first proven by Prony in 1795~\cite{prony1795essai}. 

\begin{lemma} 
\label{Lemma:Prony}
Let $N$ be an even integer.\footnote{The statement and proof can be adjusted to account for odd $N$~\cite{beylkin2005approximation}.} Let $R_m$ be as in~\cref{Lemma:Exponential_sum_Fourier_coeffs} and let $H_{R_m}$ be an $(N/2+1)  \times (N/2+1) $ Hankel matrix with entries $(H_{R_m})_{k \ell} = R_m (k + \ell)$, \hbox{$0 \leq k, \ell \leq N/2$}, where $N \geq 2m$. Let \smash{$\mathscr{P}_m(z) =\sum_{j = 0}^{m} c_j z^j$} be a polynomial with  roots \smash{$z_k= e^{\alpha_k}$}, $1 \leq k \leq m$. Then, $ {\rm rank}( H_{R_m} ) = m$, and  the null space of $H_{R_m}$ is spanned by $\{ c, Sc, \ldots, S^{N/2-m}c\},$
where \smash{$c = (c_0, \ldots, c_m, 0, \ldots, 0)^T$} and $S$ is the forward shift matrix. 
\end{lemma}   
\begin{proof} See~\cite[Lem.~2.1]{potts2011nonlinear}.
\end{proof}
  
In our case, we seek  $R_m \approx \mathcal{F}^{-1}(f)$. \Cref{Lemma:Prony} indicates that this is equivalent to  finding a Hankel matrix $H_{R_m}$ of rank $m$ so that \hbox{$H_{R_m} \approx H_v$}, where  \hbox{$(H_v)_{jk} = \hat{f}_{j+k}$,} \hbox{$0\leq j,k \leq N/2$}. Moreover, it shows that the complex exponentials in~\cref{eq:Exp_Sum} are the roots of a special polynomial, often referred to as Prony's polynomial~\cite{potts2013parameter}, whose coefficients form a vector in the null space of $H_{R_m}$. The regularized Prony's method (RPM), described in pseudocode in~\cref{alg:apm}, finds a vector $c$ of polynomial coefficients in the numerical null space of $H_v$, i.e., $c$ such that $\|H_v c\|_2 \leq \epsilon$. Unlike in~\cref{Lemma:Prony}, $H_v$ is not exactly of rank $m$, so the polynomial $\mathscr{P}(z)$ with monomial coefficients given by the entries of $c$ generally has $N/2$ roots. Since we only want decaying exponentials in $R_m$, we keep only the $m$ roots with modulus $< 1$. The exponents $\{\alpha_j\}_{j = 1}^{m}$ for $R_m$ are determined from these roots, and a least squares fit to a subset of  Fourier coefficients of $f$ supplies the weights $\{\omega_j\}_{j = 1}^{m}$ (see \cref{alg:apm}).

A qualitative error bound in terms of the singular values of $H_v$ and more details about the RPM are given in~\cite{beylkin2005approximation}. In some settings, such as in the example in \cref{fig:whales}, a good choice for the tolerance parameter $\epsilon$ may be unclear. In this case, we modify \cref{alg:apm} so that $\epsilon$ is chosen automatically by detecting gaps in the small singular values of $H_v$ that indicate the presence of a numerical null space. \Cref{alg:apm} naively implemented has an $\mathcal{O}(N^3)$ cost because it requires finding the singular value decomposition (SVD) of $H_v$. This is improved if one finds the SVD with an algorithm that takes advantage of fast matrix-vector products for Hankel matrices (e.g., the randomized SVD~\cite{halko2011finding} and Lanczos-based methods~\cite{golub2012matrix}). 

\begin{algorithm}[t!]
\caption{The regularized Prony method.}
\label{alg:apm}
\begin{algorithmic}
\STATE \hspace{-.45cm} 
\STATE \textbf{Input:}  tolerance parameter $\epsilon$ and Fourier coefficients $v = (\hat{f}_0, \ldots, \hat{f}_N)^T$. 
\STATE \textbf{Output:} $\{(\omega_j, \alpha_j)\}_{j = 1}^m$ defining $R_m$ in~\cref{eq:Exp_Sum}, so that $|R_m(j) -\hat{f}_j|\approx \epsilon$. 

\vspace{.2cm} 

\STATE 1. Construct the Hankel matrix $H_v$, where $(H_v)_{j+k} = \hat{f}_{j+k}$ for $0 \leq j,k \leq N/2$. 
\STATE 2. Compute the SVD of $H_v$ to find $c$, where $\|H_v c \|_2 \leq \epsilon$, $\|c\|_2 = 1$. 
\STATE 3. Set $\mathscr{P}(z) = \sum_{\ell = 0}^{N/2} c_{\ell} z^\ell$.
\STATE 4. Find the  $m \leq N/2$ roots $\{z_j\}_{j =1}^{m}$ of $\mathscr{P}(z)$ with $|z_j|<1$. Set $\alpha_j = \log z_j $. 
\STATE 5. Compute the least-squares solution to $V\omega = v$, where $V_{jk} = z_{k+1}^{j}$, \hbox{$0\leq j\leq N$}, \hbox{$0 \leq k \leq m\!-\!1$}.\end{algorithmic}
\end{algorithm}

 \subsubsection{The regularized Prony method as a filter} 
 In practice, one expects that samples of $f$ are corrupted by noise. The RPM has a natural interpretation as a type of filter. Rather than, for example, filtering out the high-frequency components of a signal, it separates a signal into the sum of two parts by splitting the Hankel matrix $H_v$  into the sum $H_v = H_{R_m} + H_{\mathcal{N}}$. The first term encodes a sequence of coefficients that are well--approximated in Fourier space by a length $m$ sum of exponentials (and thus correspond to a trigonometric rational). The second term encodes a sequence of coefficients that are not well approximated by such an expression. This is referred to as an annihilating filter in the literature on signals with so-called finite rates of innovation~\cite{vetterli2002sampling}.  The example in \cref{fig:whales} displays noisy data collected by a hydrophone. The noise is not well represented by low to moderate degree trigonometric rationals, so this structured low rank approximation process filters it out.    
 
 \begin{figure}
\centering
    \begin{minipage}{.45\textwidth} 
 \centering
  \begin{overpic}[width=\textwidth]{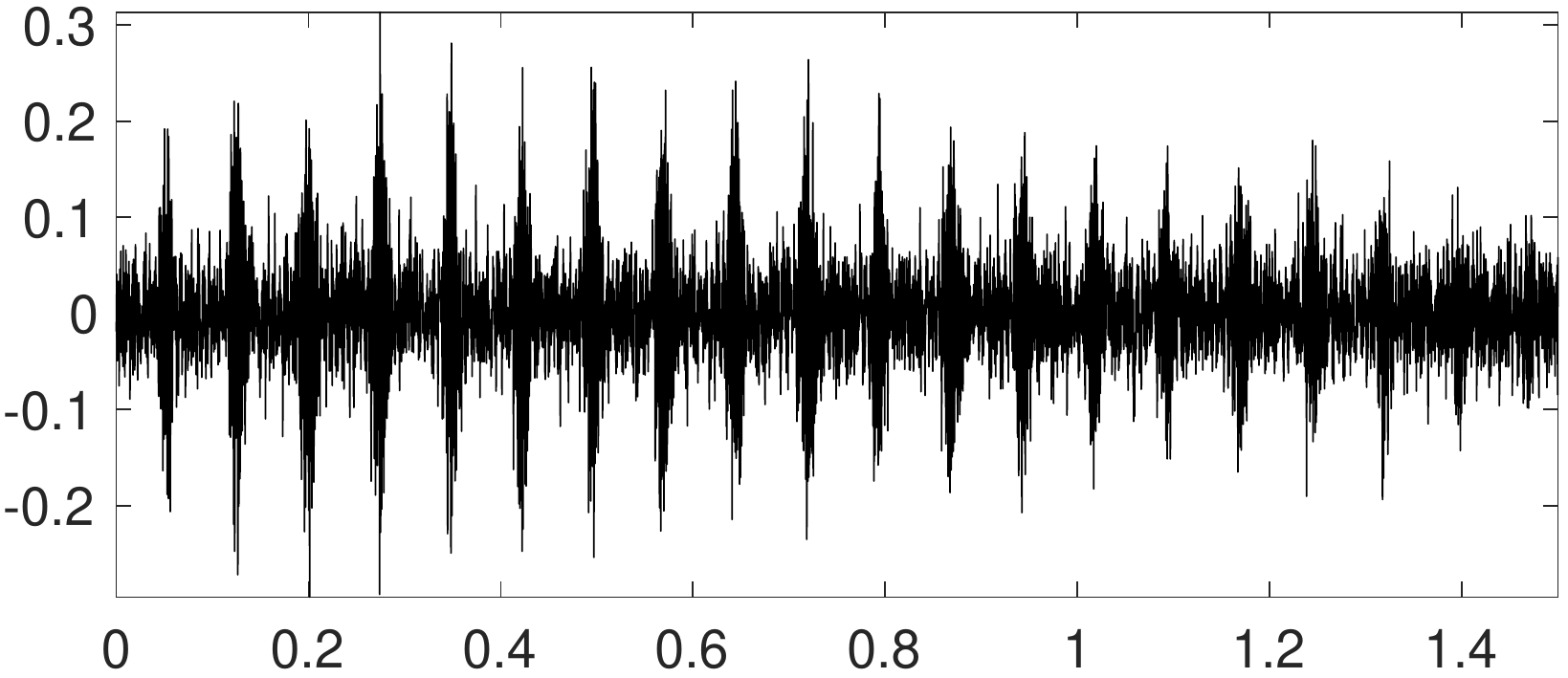}
  \end{overpic}
  \end{minipage}
  \centering
      \begin{minipage}{.45\textwidth} 
 \centering
  \begin{overpic}[width=\textwidth]{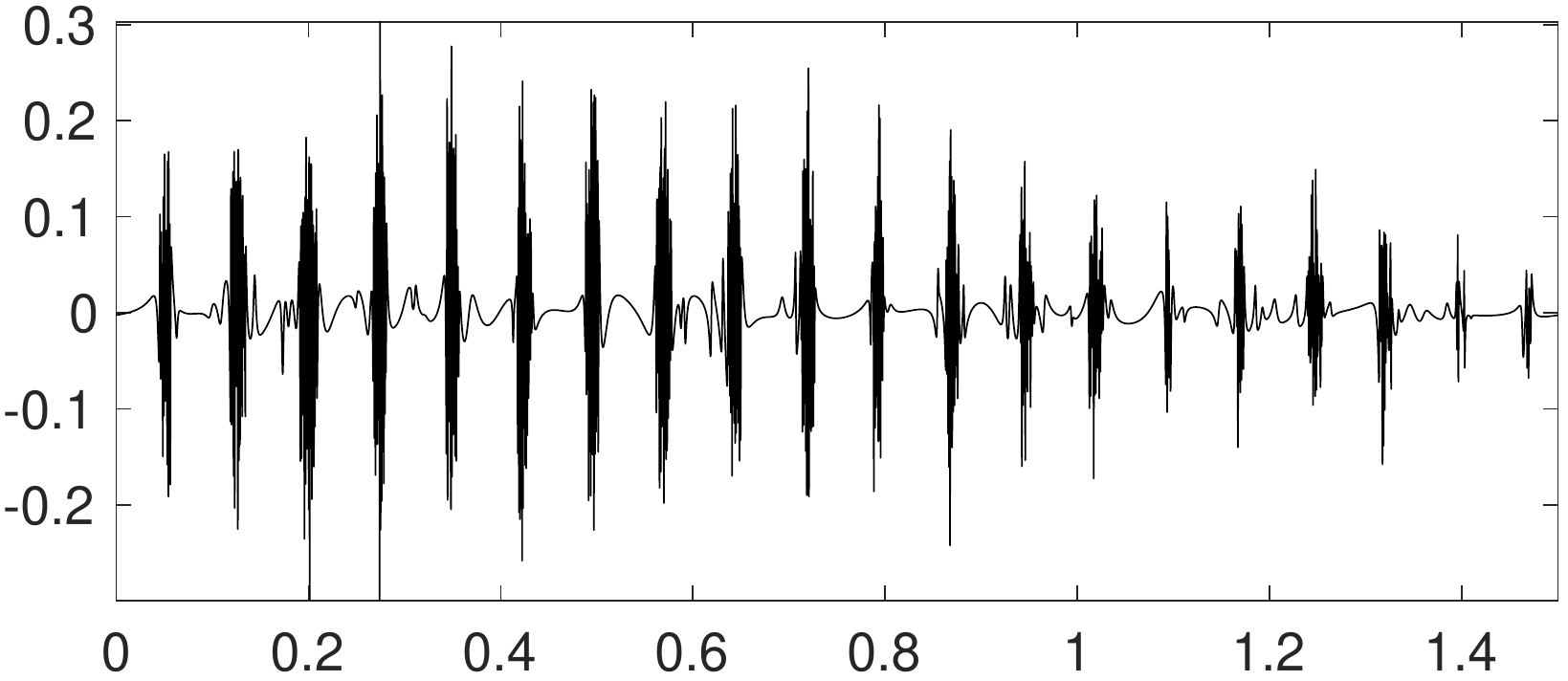}
  \end{overpic}
  \end{minipage}
  \caption{Left: A noisy recording of the trill portion of a Pacific blue whale's song. The sample consists of $6001$ equally spaced observations recorded over $1.5$ seconds~\cite{MATLABsig}. Right: A type $(274, 275)$ trigonometric rational approximation to the signal. The approximant is constructed by applying \cref{alg:apm} directly to the data with the tolerance parameter  $\epsilon = (2 \times 10^{-4}) \|H_v\|_2$. We additionally omit any terms in $R_m$ with weights smaller than $\epsilon$. The RPM filters out highly oscillatory noise that is not well-captured by trigonometric rationals, making it easier to identify the time-localized pulses in the trill. Once the approximant is constructed, one can toggle between a barycentric representation and the RPM-constructed representation as a sum of complex exponentials to perform various postprocessing tasks (see \cref{tab:fun}).} 
  \label{fig:whales}
  \end{figure}

\section{Fourier and inverse Fourier transforms} \label{sec:transforms}
The RPM and pronyAAA automatically construct compressed representations of  $f$, but these representations are very different from one another. This section describes Fourier/inverse Fourier transforms that allow us to move between these representations. If $R_m$ is a length $m$ sum of exponentials, the existence of a trigonometric rational $r_m = \mathcal{F}^{-1}(R_m)$ is guaranteed by~\cref{Lemma:Exponential_sum_Fourier_coeffs}. However, the lemma does not reveal if or how $r_m$ can be expressed in barycentric form. In the same way, given a trigonometric barycentric interpolant $r^{\gamma, t}_m$, it is not clear from~\cref{eq:BarycentricRat} how one can recover the sum of exponentials $R_m = \mathcal{F}(r^{\gamma, t}_m)$. 
In fact, the exact recovery of one representation from the other is an ill-conditioned problem.  With this in mind, we develop lossy but stable transform routines. In the REfit software~\cite{wilberREfit}, these transforms are accessed with the commands \texttt{ft} and $\texttt{ift}$. 

\subsection{ The forward transform} \label{sec:FT}
Given  $r_m^{\gamma, t}$ as in~\cref{eq:BarycentricRat}, we seek 
$$R_m(k) = \mathcal{F}(r_m^{\gamma, t}) (k) = \sum_{j = 1}^{m} \omega_j e^{\alpha_j k}, \quad k \geq 0.$$  The parameters of interest can be expressed explicitly  in terms of the poles and residues of $r_m^{\gamma, t}$. For each $j$, $\alpha_j = 2 \pi i \eta_j$ and \smash{$\omega_j = e^{-\eta_j}{\rm Res} ( r_m^{\gamma, t}(z), e^{\eta_j}),$} where 
\smash{$z = e^{2 \pi i x}$} and \smash{$\{\eta_j\}_{j = 1}^m$} are those poles of $r_m^{\gamma, t}$ with $Im(\eta_k)> 0$~\cite{beylkin2009nonlinear}.  However, using these formulas requires the accurate computation of the poles \smash{$ \eta_j$}, $1 \leq j \leq m$, and their residues. In general, this is  an ill-conditioned problem involving extrapolation off the interval of approximation. Trigonometric rationals with different pole configurations can behave almost indistinguishably on $[0, 1)$, so it is often hopeless to recover the pole locations from observations on the interval.  Known stability results depend on the poles of $r_m^{\gamma, t}$ being sufficiently well-separated from one another and the residues ${\rm Res} ( r_m^{\gamma, t}(z), e^{\eta_j})$ being bounded well away from zero~[Sec.~2]\cite{miller1970stabilized}. 
In our setting we assume that $r_m^{\gamma, t}$ is an approximation to a function $f$ that has algebraic singularities. Good resolution of these features is possible precisely because $r_m^{\gamma, t}$ has poles that cluster up near the singularities (see \cref{fig:pole_plot}). For this reason, we do not expect that one can compute the pole locations or residues with high accuracy.  The exact recovery of the parameters of $R_m$ from samples is a related problem that is also ill-conditioned~\cite{potts2013parameter, transtrum2010nonlinear}. 

Instead of trying to recover $R_m$  exactly, we apply a regularization that finds \hbox{$\tilde{R}_{\tilde{m}} \approx R_m$}, where $\tilde{m} \leq m$.\footnote{If one allows for some of the weights in $R_m$ to be zero, then one can construct sums of exponentials where it is always true that $\tilde{m} = m$.}  The poles of $r_m^{\gamma, t}$ can be approximately computed by solving a $(2m \!+\! 1) \times (2m \!+\! 1)$ generalized eigenvalue problem (see \cref{sec:Roots}). 
Suppose $\{\tilde{\eta}_k\}_{k = 1}^m$ are the computed poles with $Im(\tilde{\eta}_k)> 0$. We set $\tilde{\alpha}_k = 2 \pi i \tilde{\eta}_k$. Then, instead of computing $\{\omega_1, \ldots, \omega_m\}$ using the explicit formula,  we find a vector of weights $\tilde{\omega} = (\tilde{\omega_1}, \ldots, \tilde{\omega}_m)^T$ by solving the overdetermined  linear system $V_{\tilde{\alpha}} \tilde{\omega} = \hat{r}$, where $\hat{r}= [(\hat{r}_{m}^{\gamma, t})_{0} \cdots  (\hat{r}_{m}^{\gamma, t})_{M}]^T$ is a vector of Fourier coefficients of $r_m^{\gamma, t}$, and \hbox{$(V_{\tilde{\alpha}})_{j,k} = e^{\tilde{\alpha}_{k+1} j }$}, $0 \leq j \leq M$, $0 \leq k \leq m\!-\!1$. 
All of the Fourier coefficients of $\tilde{r}_{\tilde{m}} = \mathcal{F}(\tilde{R}_{\tilde{m}})$ are exactly produced by a length $m$ sum of exponentials with exponents $\tilde{\alpha}_k = 2 \pi i \tilde{\eta}_k $. (Some of the weights in the sum may vanish, and if so we reduce the length of the sum appropriately.)  In infinite precision, we would only need $\tilde{m}$ Fourier coefficients of $\tilde{r}_{\tilde{m}}$ to solve for the weights $\tilde{\omega}$ (see \cref{Lemma:Prony}). Instead, we must fit to the coefficients of the nearby rational $r_m^{\gamma, t}$, and so apply a modest level of oversampling. We then test the accuracy of $\tilde{R}_{\tilde{m}}$  against a randomized sample of the Fourier coefficients of $r_m^{\gamma, t}$, and  systematically increase $M$ as needed. It is typically sufficient to choose $M = 2m$.

Since finding \smash{$\{e^{\tilde{\alpha}_j}\}_{j=1}^{m}$} and solving $V_{\tilde{\alpha}} \tilde{\omega}= \hat{r}$ are each $\mathcal{O}(m^3)$ operations, the cost for computing $\tilde{R}_{\tilde{m}}$ is dominated by procuring an accurate sample $\hat{r}$. This is done first by evaluating $2 N_\epsilon +1$ samples of $r_m^{\gamma, t}$, where $N_\epsilon$ is the $\epsilon$-resolution of $r_m^{\gamma, t}$, on an equally spaced grid, and then applying an FFT.  By default, $\epsilon$ is taken to be near machine precision, and $N_\epsilon$ can be approximately found automatically using, for example, an adaptation of Chebfun's \texttt{chop} algorithm~\cite{aurentz2017chopping}. In total, computing $\tilde{R}_m$ from $r_m^{\gamma, t}$  requires $\mathcal{O}(N_\epsilon \log N_\epsilon + N_\epsilon m + m^3)$ operations.

\subsection{The inverse transform} \label{sec:IFT}
We now assume $R_m(k) = \sum_{j = 1}^m \omega_j e^{\alpha_j k}$ is given, where each  $ \alpha_j$ is distinct,  $Re(\alpha_j) < 0$, $\omega_j \neq 0$, and $k \geq 0$.  We seek an efficient representation for $r_m = \mathcal{F}^{-1}(R_m)$, which is defined to be 
$$r_m(x)  = \sum_{k = -\infty}^{-1} \overline{R_m(-k)}e^{ 2 \pi i xk} + \sum_{k = 0}^{\infty} R_m(k) e^{ 2 \pi i x k}.$$

It is not always true that $r_m$ is a type $(m-1, m)$ trignometric rational function.  However,  this result does hold under the additional assumption that $R_m(0)=0$, or, equivalently,  \smash{$\int_{0}^1r_m (x) {\rm d}(x)= 0$}. We take this to be the case, and our objective is to construct a barycentric interpolant $r_m^{\gamma, t}$ to  $r_m$. We show in the next lemma that for any set of distinct points $t = \{t_1, \ldots, t_{2m}\} \subset [0, 1)$, there is $\gamma$ such that $r_m^{\gamma, t} = r_m$.

 \begin{lemma}
\label{lemma:Interp_recovery}
Let $r_m$ be a  type $(m\!-\!1, m)$ trigonometric rational function with  simple poles that is real-valued, continuous, and periodic on $[0, 1)$.  Let $t \subset [0, 1)$ be a set of $2m$ distinct interpolating points. Then, there is a set of weights $\gamma$ such that the trigonometric barycentric interpolant $r_m^{\gamma, t}$   recovers $r_m$ exactly. 
\end{lemma}

\begin{proof} 
Consider the denominator $q_m$ in $r_m = p_{m-1}/q_m$, and assume that $q_m$ has no shared zeros with $p_{m-1}$. Since $q_m$ is a trigonometric polynomial, we can write it in barycentric form with respect to the interpolating points in $t$: 
\begin{equation} 
q_m(x)  = \ell_t(x)  \sum_{j = 1}^{2m} w_j q_m(t_j) \cot(\pi (x-t_j)), \qquad \ell_t(x) = \prod_{j = 1}^{2m} \sin(\pi(x-t_j)),
\end{equation}
where \hbox{$w_j = 1/ \prod_{k = 1, j \neq k }^{2m} \sin(\pi(t_k-t_j))$} are the polynomial barycentric weights~\cite{berrut2004barycentric} associated with $t$.  By setting $\gamma_j = q_m(t_j) w_j$ and $f_j = r_m(t_j)$, we have via~\cref{eq:BarycentricRat} that there is $r_m^{\gamma, t}(x) = n(x) \ell_t(x)/ q_m(x)$ for some function $n$, where for each $j$, $r_m^{\gamma, t}(t_j) = r_m(t_j)$.   We must now  show that $n(x) \ell_t(x) = p_{m-1}(x)$. 

The  barycentric trigonometric polynomial interpolant to $p_{m-1}$ on $t$ exists and is given by $p_{m-1}(x) = \ell_t(x) \sum_{j = 1}^{2m} w_j p_{m-1}(t_j) \cot( \pi (x - t_j))$. Expanding this in exponential form, we have that $ p_{m-1}(x) = c_{m} e^{ 2 \pi i mx} + \ldots + c_{-m}  e^{-2 \pi i mx}$, where
$$  c_{m} = \frac{1}{4i} {\rm exp}\left( { -\pi i \sum_{j = 1}^{2m} t_j }\right) \sum_{j = 1}^{2m} w_j p_{m-1}(t_j) , \quad c_{-m} = - {\rm exp}\left( { 2 \pi i \sum_{j = 1}^{2m} t_j }\right) c_m.$$ 
Since $p_{m-1}$ is of degree $m\!-\!1$,  $c_m = c_{-m} = 0$, so \smash{$\sum_{j = 1}^{2m} w_j p_{m-1}(t_j) = 0$}. This implies that $\sum_{j = 1}^{2m} \gamma_j r_m(t_j )= 0$, as  $r_m(t_j) = p_{m-1}(t_j)/q_{m}(t_j)$.  Now it is clear that $n(x) \ell_t(x)$ is  also a trigonometric polynomial of degree $m\!-\!1$. Since $n \ell_t$ interpolates  $p_{m-1}$ at $2m$  points, they must agree everywhere. \end{proof}

The computation of $\gamma$ as in \cref{lemma:Interp_recovery} via polynomial barycentric weights and  the evaluation of $q_m$ is numerically unstable except in very special cases~\cite{trefethen2013approximation}.  The stable computation of $\gamma$ and subsequently, the  error $\| r_m - r_m^{\gamma, t}\|_\infty$, depends strongly on the choice of $t$. Some of the more obvious methods for selecting the nodes perform poorly and lead to instabilities in the form of spurious poles. The following discussion is somewhat technical, but it introduces an effective heuristic for choosing a ``good" set of barycentric nodes and then stably constructing $r_m^{\gamma, t} \approx r_m$. 

\subsubsection{A modified pronyAAA for rational recovery} \label{sec:CPQR}

A simple strategy for choosing nodes is to evaluate $r_m$  on a fine enough grid and then apply $m$ steps of pronyAAA to construct the interpolant $r_m^{\gamma, t}$. This method does not usually exactly recover $r_m$ (see the discussion of exact recovery in \cref{sec:FT}), and it can be the case that the error $\|r_m^{\gamma, t} -r_m\|_\infty$ is unacceptably large. A few additional steps of pronyAAA may drive the error down, though this results in a trigonometric rational interpolant with more poles than $r_m$.  However, a more pernicious problem with this approach is that demanding accuracy close to machine precision from AAA-based methods can result in spurious poles on the interval of approximation that cannot be eliminated without adversely impacting accuracy~\cite{nakatsukasa2018aaa}. 

 To avoid introducing spurious poles, we use the poles of  $r_m$, which are known explicitly from $R_m$ via~\cref{Lemma:Exponential_sum_Fourier_coeffs}. There is no hope of exactly preserving the poles. However, if $m$ is fixed and $r_m^{\gamma, t}$ is constructed such that it approximately preserves the given poles, then it cannot also admit arbitrary spurious poles. 
This motivates a three-step procedure for constructing $r_m^{\gamma, t}$ that mixes a pole-preservation strategy involving a type $(m\!+\!K-\!1, m \!+\! K)$ trigonometric rational with a data-driven strategy:
\begin{itemize}
\item[(1)] A candidate set $\tilde{t}$ of $2m + 2K$ barycentric nodes is chosen, where $K \geq 0$ is an oversampling parameter.  Subsets of $\tilde{t}$ admit type $(m\!-\!1, m)$ barycentric trigonometric interpolants with poles close to those of $r_m$. 
\item[(2)] The interpolant \smash{$r_{m+2K}^{\tilde{\gamma}, \tilde{t} }$} is constructed via a pole-preserving least-squares fit to samples of $r_m$ (discussed below), so that it has $2m$ poles close those of $r_m$. 
\item[(3)] The pronyAAA cleanup procedure (see \cref{cleanup}) is applied to remove the $2K$ poles of \smash{$r_{m+2K}^{\tilde{\gamma}, \tilde{t}}$} with the smallest residues. This selects $t$, a set of $2m$ barycentric nodes, from $\tilde{t}$. The barycentric weights $\{\gamma_1, \ldots, \gamma_{2m}\}$ are then computed via~\cref{eq:AAAopt}. Note that the poles of $r_m^{\gamma, t}$ must be recomputed. 
\end{itemize}
A version of this method without oversampling (i.e., with $K = 0$, $\tilde{t} = t$, and $\tilde{\gamma} = \gamma$) is useful for motivating how the barycentric nodes in Step (1) are selected. In such a setting, Step (2) simplifies substantially and Step (3) is not needed. However, it is more stable to choose $K > 0$, and in practice, we usually take $K = 1$.  We first describe the $K \!=\! 0$ case, and then use it to explain the method for $K > 0$. 

 \textbf{Case 1: $K = 0$.}  Suppose that $T$, the discretization of $[0, 1)$ from which  $t$ is chosen, consists of points $ x_0 < x_1< \ldots < x_{2N}$. Let  $P = \{\eta_1, \ldots, \eta_{2m}\}$ be the poles of $r_m$. Ideally,  $r_m^{\gamma, t}$ can be constructed so that its  poles  are given by $P$.  Noting that the poles of $r_m^{\gamma, t}$ are the zeros of the denominator polynomial \smash{$d_m(x) = \sum_{j = 1}^{2m}\cot(\pi (x -  t_j))$} in \cref{eq:BarycentricRat}, we introduce the matrix \hbox{$D_T \in \mathbb{C}^{(2m\!+\!1) \times (2N\!+\!1)}$}:
\begin{equation} 
\label{eq:Polematrix}
D_T = \left[ \begin{array}{cccc}
 \ell_{1,0}& \cdots & \cdots & \ell_{1,2N} \\
    \vdots &  &  & \vdots \\
    \ell_{2m, 0} & \cdots &  \cdots & \ell_{2m,2N} \\
    \hline
    r_m(x_0) & \cdots &  \cdots & r_m(x_{2N})
\end{array} \right],   \quad \ell_{j,k} = \cot( \pi( \eta_j -  x_k) ).
\end{equation} 
Using $D_T$, we relate the selection of barycentric nodes to a column subset selection problem. Indexing from $0$, denote by $(D_T)_{k}$ the $k$th column of $D_T$. The $k$th column is associated with the point $x_k$ in $T$. The set of nodes \hbox{$t  = \{x_{k_1}, x_{k_2}, \ldots x_{k_{2m }}\}$} then corresponds to a set of columns that forms the submatrix \hbox{$D_t = \left[ (D_T)_{k_1}, \ldots, (D_T)_{k_{2m}} \right ]$}.  From \cref{lemma:Interp_recovery}, there is $\gamma = (\gamma_1, \ldots, \gamma_{2m})^T$ such that $r_m^{\gamma, t} = r_m$, and $\gamma$ is in the null space of $D_t$. The first $2m$ entries of $D_t\gamma$ are evaluations of $d_m$ at its zeros. The last entry of $D_t\gamma$ is also zero, since numerator of $r_m^{\gamma, t}$ is of degree $m\!-\!1$ (see \cref{sec:bary}).   If  $\gamma$ can be computed from $D_t$ accurately, then clearly $t$ is an excellent set of interpolating points. However, the accuracy of this computation depends on properties of $D_t$. In particular, there are stable ways to compute $\gamma$  if $2m\!-\!1$ of the columns of $D_t$ form a well-conditioned matrix~\cite{gu1996efficient}. 

This suggests that we choose the points $t$ by choosing a subset of columns from $D_T$ that are close to orthogonal. Several kinds of rank-revealing algorithms can be applied to $D_T$ to approximately solve this problem, including the column-pivoted QR (CPQR) algorithm. This  constructs the factorization $D_TP  = Q_T R$, where $P$ is a permutation matrix, and the leading $\ell \leq {\rm rank}(D_T)$ columns of $D_TP$ have been greedily selected to minimize their linear dependence on one another~\cite[Sec.~6.4]{golub2012matrix}. As a consequence of \cref{lemma:Interp_recovery}, any submatrix of $D_T$ consisting of $2m$ or more columns is rank-deficient, so ${\rm rank}(D_T) \leq 2m\!-\!1$.\footnote{If we assume that $r_m$ has a denominator of exactly degree $m$, then ${ \rm rank}(D_T) = 2m\!-\!1$.} We choose $2m\!-\!1$ points in $t$ by performing CPQR on $D_T$.  In principle, the final point in the set $t$ should be chosen so that the accuracy of the computed right singular vector in the nullspace of $D_t$ in~\cref{eq:Polematrix} is maximized. Instead, we choose the point associated with the column in the trailing $(|T|\!-\!2m\!+\!1)$  columns of $D_TP$  that has the smallest 2-norm. Though they are selected quite differently, these CPQR-selected barycentric nodes  concentrate around singularities, just like the nodes selected by pronyAAA (see \cref{sec:CPQRvsBary} and \cref{fig:CPQRclusters}).  
 
If one can compute the vectors in the null space of $D_t$ accurately, then $\gamma$ can be taken as the last right singular vector of $D_t$. However, this is rarely the case. The accurate recovery of $\gamma$ from $D_t$ can be problematic even with the best possible choice $t \subset [0, 1)$. For this reason, we require a strategy that additionally incorporates a fit to samples. The simplest idea is to use the $2m$ points as selected above and then find the barycentric weights via \cref{eq:AAAopt}. However, this strategy does not seem to eliminate spurious poles or reduce error as effectively as the procedure we describe below.
  
 \textbf{Case 2: $K \neq 0$.} 
Typically, it is sufficient to take $K\! =\! 1$, though one can also choose a larger $K$. Construct an oversized candidate set $\tilde{t} = \{ x_{k_1}, \ldots, x_{k_{2m+2K}} \}$ using the CPQR-based method from Case 1. Then $D_{\tilde{t}}$ is of size \hbox{$(2m+1) \times (2m+2K)$} and has a numerically detectable null space.  We compute the barycentric weights of the interpolant $r_{2m+2K}^{\tilde{t}, \tilde{\gamma}}$ by requiring that $\tilde{\gamma} =  \tilde{Q}\eta$, where the columns of $\tilde{Q}$ are orthogonal and approximately span the null space of \smash{$D_{\tilde{t}}$}. We select $\eta$  to minimize \smash{$ \|C\tilde{Q}\eta\|_2$}, with $C$ constructed as in  \cref{eq:linearizedOpt}.  Approximate poles and residues of \smash{$r_{2m+2K}^{\tilde{t}, \tilde{\gamma}}$} can then be computed in $\mathcal{O}(m^3)$ operations (see \cref{sec:Roots}). It is almost always the case in practice that $2K$  poles of \smash{$r_{2m+2K}^{\tilde{t}, \tilde{\gamma}}$} are negligible in that they have residues with tiny magnitudes. With this in mind, we sort the poles by the magnitude of their residues. As in the pronyAAA cleanup routine, for each of the $2K$ poles with the smallest residues, we eliminate the point in $\tilde{t}$ that is nearest to the pole. The remaining points in $\tilde{t}$ are taken as $t$, and the set of barycentric weights are found as in a standard step of pronyAAA, i.e., as the minimizer of~\cref{eq:AAAopt}. 
 
 This strategy first selects a set of interpolating points for which an interpolant with good properties (e.g.,  poles off $[0, 1)$) is known to exist, and then fits the interpolant to samples of $r_m$. 
We remark that this is a heuristic. There is no guarantee that spurious poles are avoided or the original poles of $r_m$ are preserved. It remains unclear why the solution in Step (3) often inherits the good pole properties associated with the initial solution in Step (2), and under what circumstances this inheritance can be assured. Nonetheless, we find that the method works extremely well in many cases where simply applying pronyAAA fails.  
  
\paragraph{Implementational details}
In practice, we start with $K = 0$.  When $\gamma$ can be recovered with high accuracy directly from $D_t$ , we recover it and end the procedure. This can be checked by computing the singular values of $D_t$ or by using estimates related to the CPQR routine~\cite{gu1996efficient}. 
When this isn't possible, we set $K = 1$ and enlarge our selection of candidate barycentric nodes, which requires no additional computation. Then, we move on to Steps (2) and (3).  If the method fails and spurious poles are detected, we first try enlarging $\tilde{t}$ by setting $K = 2$ and trying again. If this fails, it can often be remedied by resampling $r_m$ on a denser grid and starting over at Step (1). When resampling does not solve the issue, we instead construct a stable barycentric interpolant using pronyAAA by accepting a lower level of accuracy. 

\begin{figure}
\centering
\hspace{.2cm}
      \begin{minipage}{.45\textwidth} 
 \centering
  \begin{overpic}[width=\textwidth]{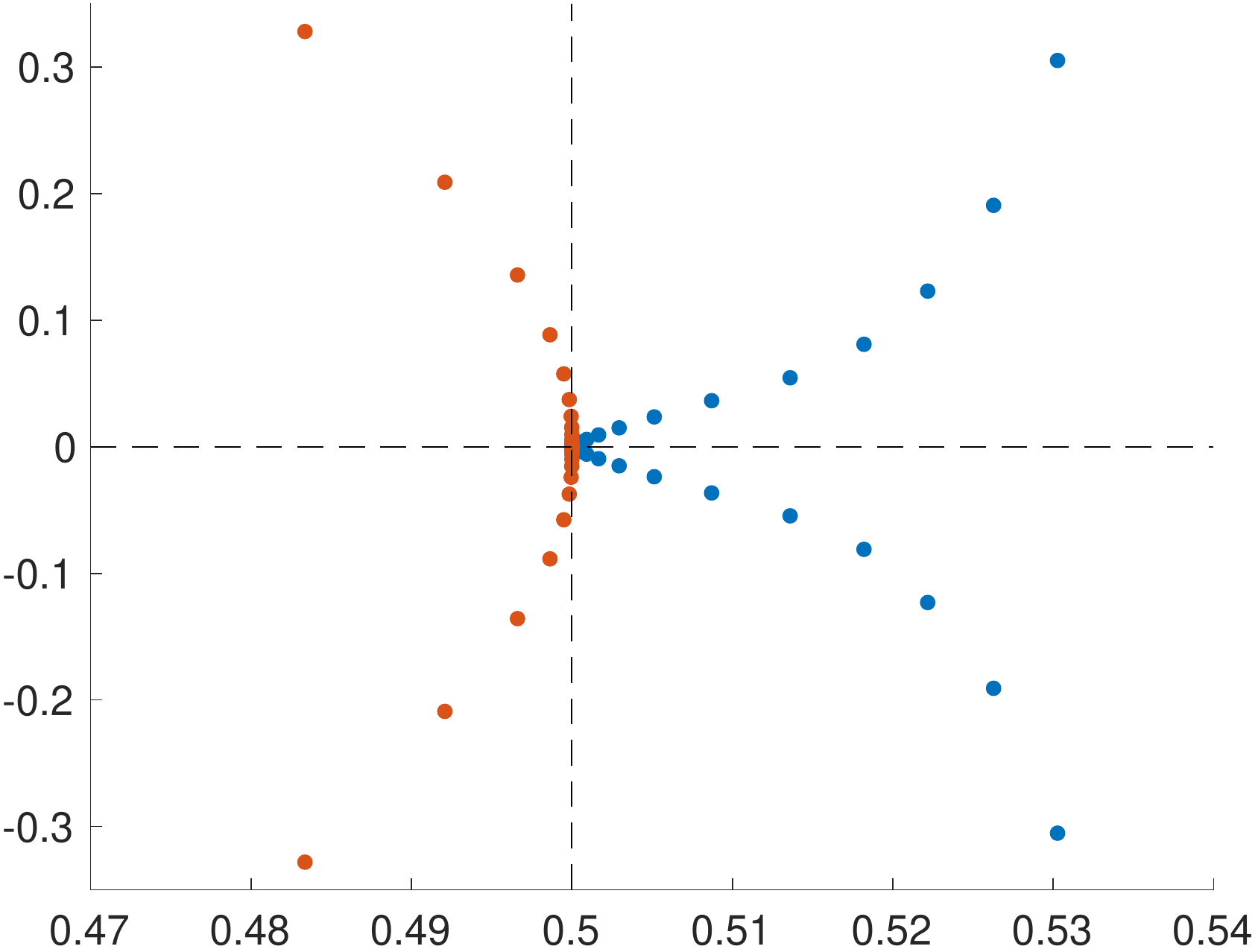}
     \put(95, 39){\small{$Re$}}
      \put(42, 75){\small{$Im$}}
  \end{overpic}
  \end{minipage}
    \hspace{.2cm}
    \begin{minipage}{.45\textwidth} 
 \centering
  \begin{overpic}[width=\textwidth]{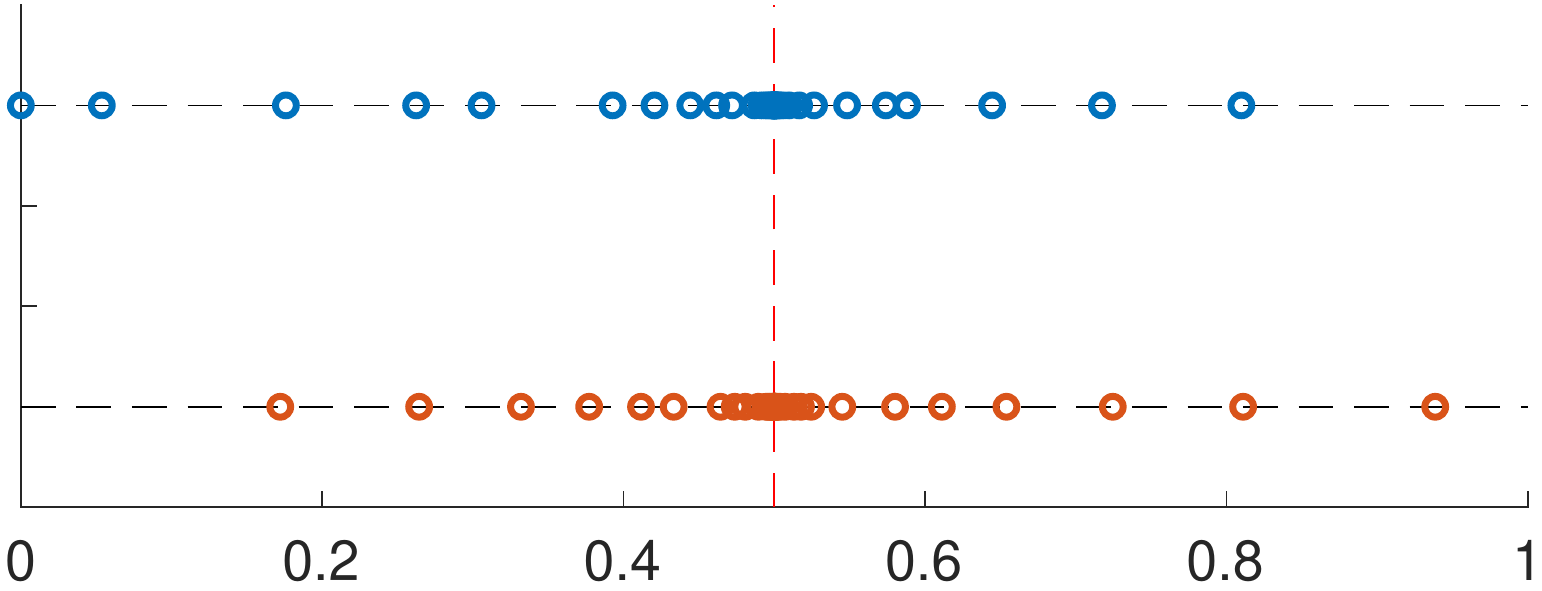}
    \put(50, -5){\rotatebox{0}{\small{$x$}}}
  \end{overpic}
  
  \vspace{.5cm}
    \begin{overpic}[width=\textwidth]{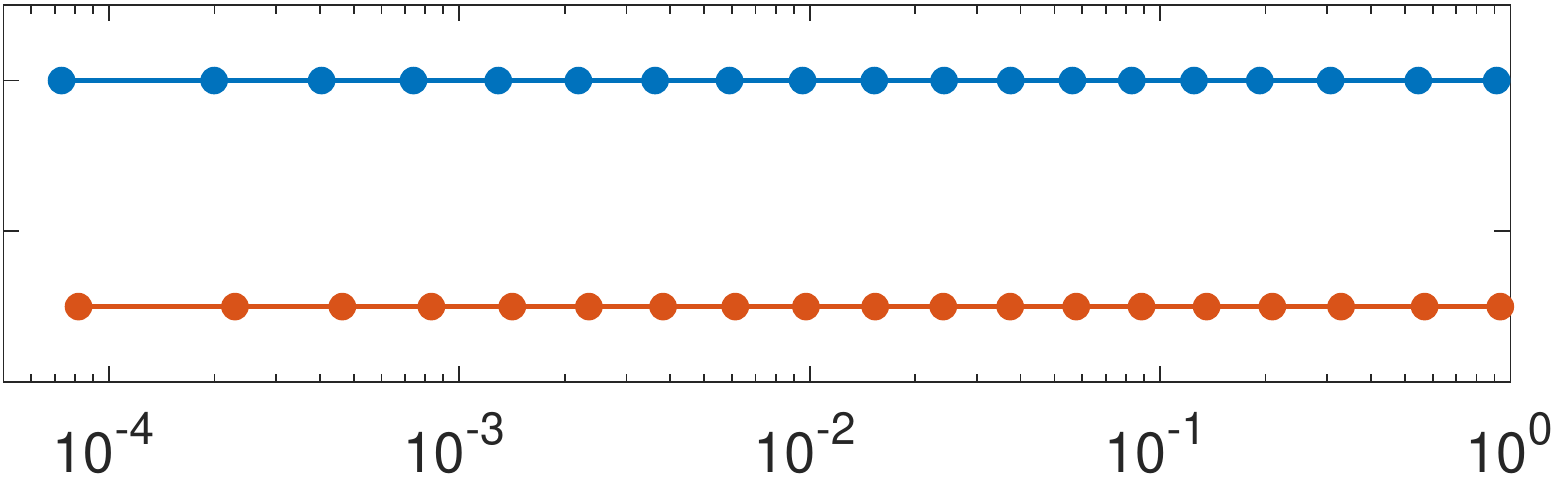}
    \put(35, -6){\rotatebox{0}{\small{$d_j = |\eta_j -.5|$}}}
  \end{overpic}
  \end{minipage}
  \vspace{.3cm}
  \caption{Left: poles of \smash{$r_{bary}$} (blue) and \smash{$r_{ift}$} (orange) are plotted in the complex plane. Here, \smash{$r_{bary}$} is a type $(18, 19)$ barycentric interpolant to $f$ constructed via pronyAAA, and \smash{$r_{ift} = \mathcal{F}^{-1}(R_{19}) \approx f$},  where \smash{$R_{19}$} is an exponential sum as in~\cref{eq:Exp_Sum}, and \smash{$r_{ift}$} is constructed by applying the inverse Fourier transform from \cref{sec:CPQR} to \smash{$R_{19}$}. The function $f$ is given by $f(x) = |\sin(\pi(x-1/2))| - \pi/2$, and has a singularity at $x = 1/2$. 
  Right upper: The locations of the barycentric nodes for \smash{$r_{bary}$} (blue) and \smash{$r_{ift}$} (orange). Right lower: The distances \smash{$d_j = |\eta_j - .5|$} from the singularity, where each \smash{$\eta_j$} is a pole with \smash{${\rm Im}(\eta_j) >0$}, are sorted by size and plotted on a logarithmic scale (shown in blue for \smash{$r_{bary}$}, orange for \smash{$r_{ift}$}). This shows that both the pronyAAA-constructed interpolant and the CPQR-based interpolant have poles that exhibit the tapered-type clustering associated with quasi-optimal rational approximation schemes~\cite{trefethen2021exponential}.}
  \label{fig:CPQRclusters}
  \end{figure}
  
\subsubsection{Example: Two types of barycentric interpolants} \label{sec:CPQRvsBary}
 In \cref{fig:CPQRclusters}, we compare the properties of two types of rational approximations to the function \hbox{$f(x) = |\sin(\pi(x-1/2))| - \pi/2$}. First, we apply pronyAAA to a set of $6000$ samples of $f$ taken on an equally-spaced grid $T$. This constructs \smash{$r_{bary}$}, a type $(18, 19)$ trigonometric rational, where away from the singularity, $|f(x) - r_{bary}(x)|\approx 10^{-8}$.  The locations of the barycentric nodes selected by pronyAAA are plotted (blue) in the upper right panel of \cref{fig:CPQRclusters}. In the left panel, a subset of the poles of $r_{bary}$ are plotted (blue) in the complex plane. Both the nodes and the poles  cluster up near the singularity \hbox{$x = 1/2$}. Shown in red in the same plots are the CPQR-selected barycentric nodes from \cref{sec:CPQR}, and the poles of the barycentric trigonometric rational \smash{$r_{ift} = \mathcal{F}^{-1}(R_m)$}, where $m = 19$. Here, \smash{$R_m$} is an exponential sum constructed via the RPM using  samples  of $f$ on $T$, and \smash{$r_{ift}$} is constructed using the procedure in \cref{sec:CPQR}.  The nodes and poles of \smash{$r_{ift}$} also cluster near $x = 1/2$, but in spatial patterns that are quite different from those of \smash{$r_{bary}$}. A closer investigation of the pole clustering patterns (see~\cref{fig:CPQRclusters}, lower right) reveals that in both cases, the sets of distances \smash{$d_1\leq d_2 \leq   \ldots \leq d_{19}$} from the singularity, where  \smash{$d_j = |\eta_j - 1/2|$} and each \smash{$\eta_j$} is a pole with \smash{${\rm Im}(\eta_j) > 0$}, have  the tapered-type spacing on a logarithmic scale that is associated with best (and near-best) convergence rates~\cite{trefethen2021exponential}.  In addition to revealing that these functions have properties associated with quasi-optimal approximation power,  the locations and clustering patterns of the nodes or poles can be used to identify singularities, extract features, or classify signals.

\section{Signal reconstruction in time and frequency space} \label{sec:Examples}
With the Fourier and inverse Fourier transforms available, we can combine the advantages of pronyAAA and the RPM to overcome various issues, such as undersampling or noise.  We illustrate this idea with two examples. Then in \cref{sec:REfit}, we describe a collection of algorithms for computing with trigonometric rational functions and exponential sums that exploits our ability to move stably between the representations.

\subsection{An undersampled function}

\begin{figure}
\centering
    \begin{minipage}{.45\textwidth} 
 \centering
  \begin{overpic}[width=.95\textwidth]{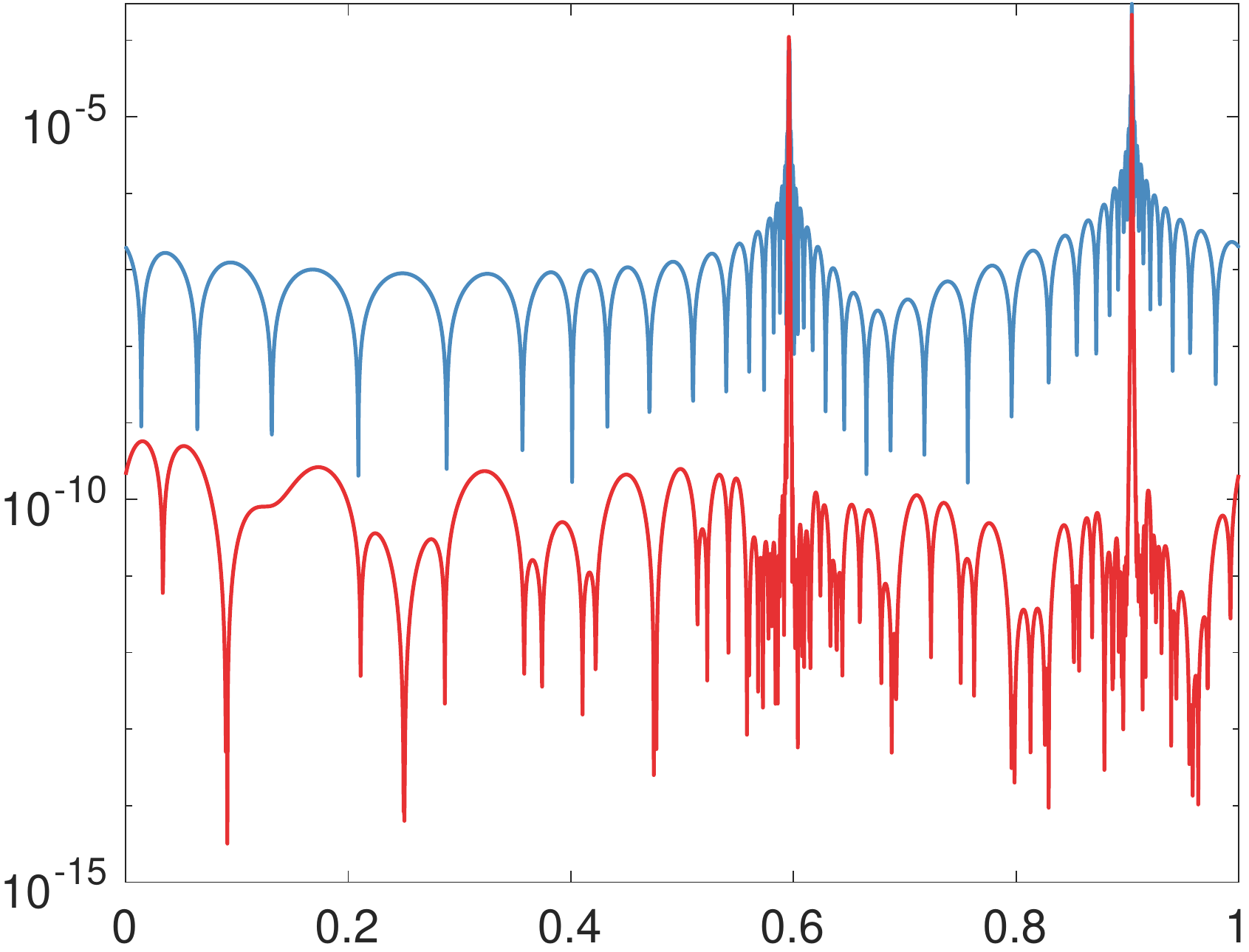}
    \put(30,80){Time domain}
      \put(-7, 35){\rotatebox{90}{error}}
    \put(13, 58){\rotatebox{0}{\tiny{ (Prony's method alone)}}}
    \put(19, 13){\rotatebox{0}{\tiny{ (pronyAAA + Fourier transform)}}}
  \end{overpic}
  \end{minipage}
      \begin{minipage}{.45\textwidth} 
 \centering
  \begin{overpic}[width=.95\textwidth]{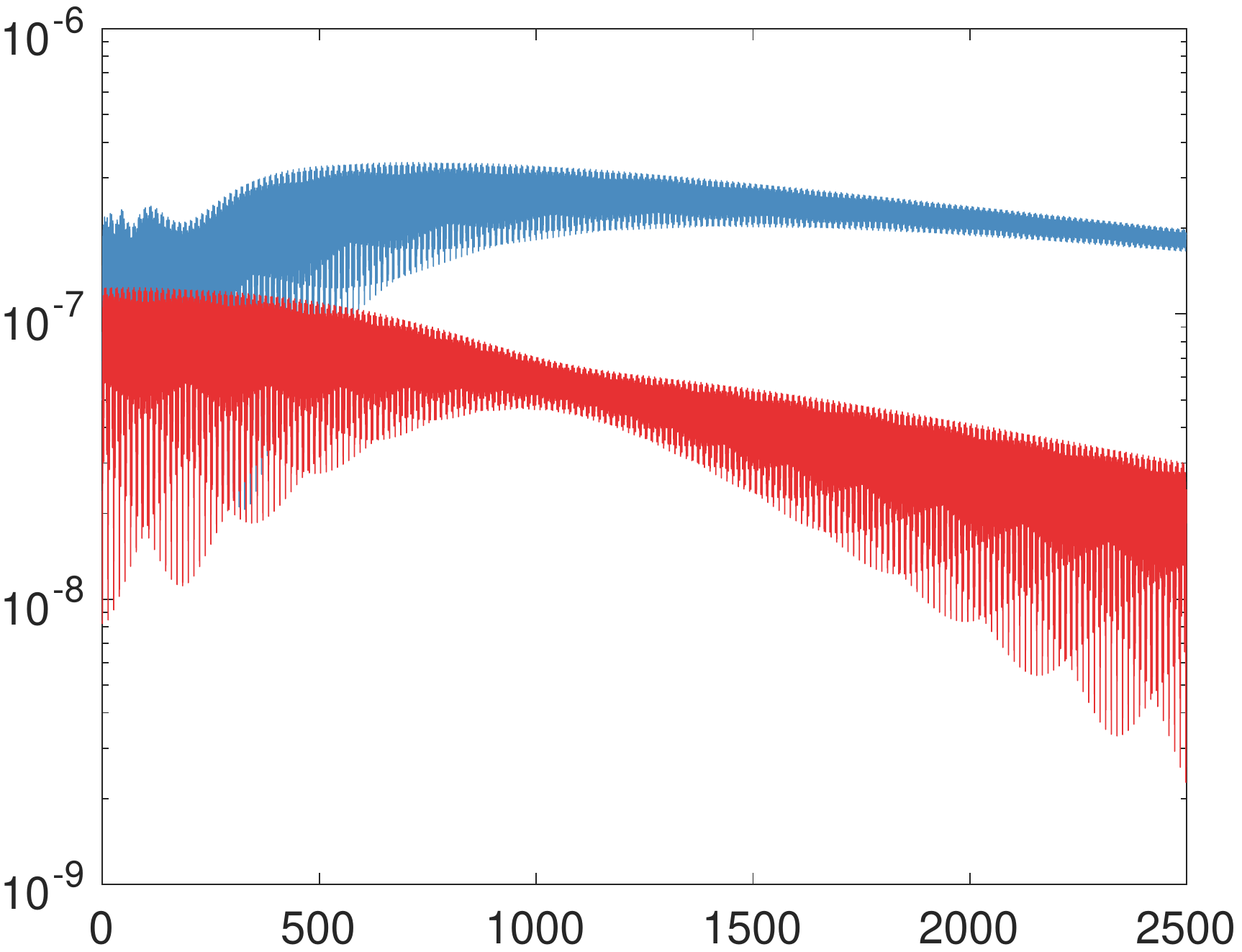}
    \put(25,80){Frequency domain}
    \put(45, 65){\rotatebox{-4}{\tiny{ (Prony's method alone)}}}
    \put(31, 53){\rotatebox{-10}{\tiny{ (pronyAAA + Fourier transform)}}}
  \end{overpic}
  \end{minipage}
  \caption{Left: The absolute error in approximating $f(x) = |\sin(2 \pi x)/4 + {\rm exp}( \sin( 2 \pi x))|/4 + c$  with two different rational approximants, $\mathcal{F}(R_b)$ (blue) and $\mathcal{F}(R_{r})$ (red), is plotted on a logarithmic scale at $3000$ equally-spaced points: $R_b$ is an exponential sum of length $14$ that was adaptively constructed via the RPM (see~\cref{alg:apm}) from a sample consisting of only $1401$ equally-spaced points. The tolerance parameter is set to $\epsilon = 10^{-10}$, but the coarseness of the sample limits the achievable accuracy of the representation. $R_r$ (red) is constructed by first applying pronyAAA in signal space to construct a barycentric interpolant $r_r^{\gamma, t}$,  and then using the Fourier transform function to compute $R_r = \mathcal{F}(r_r^{\gamma, t})$.  Right: The absolute errors in Fourier space between accurately computed Fourier coefficients of $f$ and the exponential sums $R_b$ (blue) and $R_r$ (red) are plotted on a logarithmic scale against the Fourier modes $0, 1, \ldots, 2500$. 
  }
  \label{fig:Slow_decay}
  \end{figure}
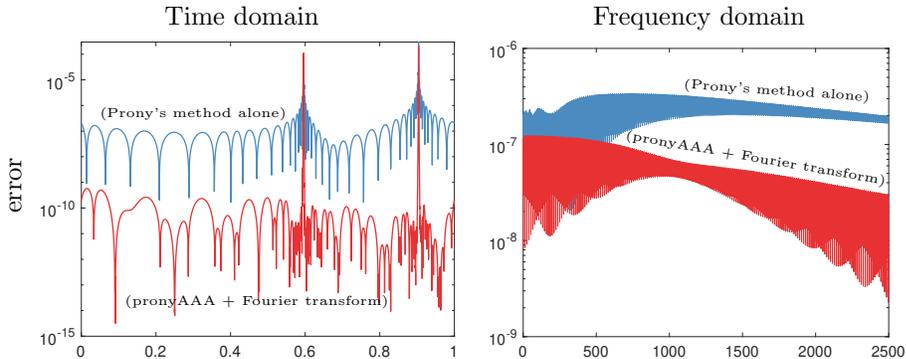
  
In this example, we consider a function $f(x) = |\sin(2 \pi x)/4 + {\rm exp}( \sin( 2 \pi x))|/4 + c$, which has Fourier coefficients that decay asymptotically like $\mathcal{O}(|k|^{-2})$, where $k$ denotes the $k$th Fourier mode. Here, $c$ is a normalization parameter ensuring that the mean value of $f$ over $[0, 1)$ is zero. We suppose that  $f$ is sampled at $1401$ equally spaced points, and that an exponential sum representing $\mathcal{F}(f)$ is desirable for downstream tasks.  The direct application of Prony's method performs poorly because $f$ is undersampled. We denote the constructed exponential sum as $R_b$. The error in the computation of the Fourier coefficients via the FFT is on the order of $10^{-6}$, so we cannot expect accuracy much better than that. An alternative approach is to apply pronyAAA to construct the barycentric interpolant $r_r^{\gamma, t}$, and then compute $R_r = \mathcal{F}(r_r^{\gamma, t})$ as in \cref{sec:FT}.

In~\cref{fig:Slow_decay} (left), we use the pole-residue format to directly evaluate the values of the rationals associated with the two types of constructed exponential sums. There is a tiny band around the two singularities where the errors incurred by the two methods are approximately the same. Elsewhere,  the accuracy achieved by first applying pronyAAA is nearly double that attained by Prony's method alone. The error in recovering the Fourier coefficients of $f$ is diffuse and more accurate, especially in the extrapolation of the tail (see~\cref{fig:Slow_decay}, right).  The exponential sum $R_r$, with only $13$ terms, represents $f$ with highly localized error behavior, and it is in a form efficient for storage, convolution, and other tasks (see \cref{sec:REfit}).

\subsection{Reconstruction of an ECG signal} 
\label{sec:ECG}

Rational approximation methods are effective in many biomedical monitoring tasks, including the processing of electrocardiogram (ECG) signals~\cite{fridli2011rational, gilian2014ecg}.  In~\cite{fridli2011rational}, rational functions constructed in the orthogonal rational Malmquist–Takenaka basis are used to reconstruct ECG signals and then classify them. The rationals perform with better overall compression properties and have several other advantages when compared to wavelets, splines, and other families of functions~\cite{kovacs2019generalized}.  We do not expect to outperform such a highly specialized scheme with our approach. However, this example illustrates that our more general-purpose method effectively constructs a denoised representation of the signal.  

In this example, we apply the RPM and fit a rational function directly to noisy ECG data taken from the PhysioNet MIT BIH arrhythmia database~\cite{moody2000physionet}.  As in~\cite{fridli2011rational}, the location of its poles can be used for classification and feature recognition tasks. Using the inverse Fourier transform function described in  \cref{sec:CPQR}, we can construct a barycentric trigonometric rational representation of the function, which is a convenient format for identifying local extrema (see \cref{sec:REfit}). This can all be done with three lines of code in REfit: 

\vspace{.5cm}

\indent \hspace{3cm}  \texttt{R = efun(data, 'tol', 1e-3); } \\
\indent \hspace{3cm} \texttt{r = ift(R); } \\
\indent \hspace{3cm} \texttt{extrema = [ min(r); max(r)]; } \\

\indent If one tries to use pronyAAA directly, the result is a trigonometric rational with $200$ poles, and the data set only contains $645$ samples. Of these poles, $62$ are spurious and lie on the interval of approximation. This happens because the pronyAAA algorithm does not distinguish between the signal and the noise, and it tries to induce a fit to noise by adding poles. A better approach is first to apply the RPM. Within the first two lines of the above code, several tasks are being executed: First, the exponential sum $R_m$ (here, $m = 35$) stored in \texttt{R} is constructed via the RPM. The RPM automatically filters out additive noise on the sample with magnitudes approximately at or below the tolerance level $\epsilon = 10^{-3}$.  Then, \texttt{R} is used to extrapolate high-frequency information that lies beyond the noise limitation (see~\cref{fig:ECG}, left). This provides an enriched sample for selecting interpolating points and constructing the barycentric interpolant \texttt{r}. The construction of \texttt{r} in this way can be viewed as a form of super-resolution~\cite{candes2014towards}.  If one tries to construct a barycentric interpolant without enriching the sample by extrapolation in frequency space, then spurious poles appear that one cannot eliminate without destroying the accuracy of the approximation. This is because the signal is not well-resolved in the time domain at the original sample rate.  Once \texttt{r} is available, one can then automatically and efficiently perform a variety of processing tasks, such as rootfinding and the detection of maxima and minima.    

\begin{figure}
 \centering
   \begin{minipage}{.43\textwidth} 
 \centering
  \begin{overpic}[width=.9\textwidth]{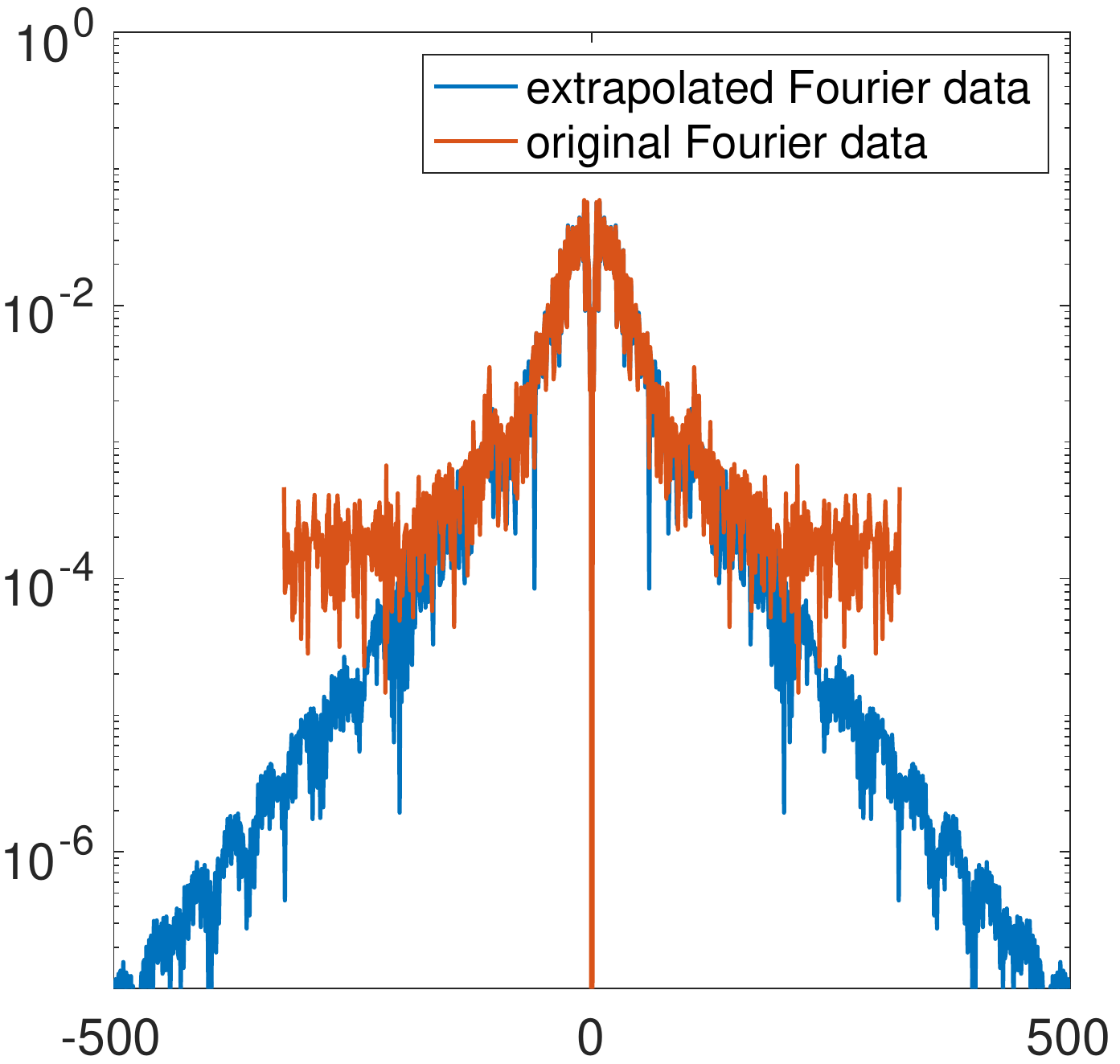}
  \put(20, 98){{Frequency domain}}
  \put(-7, 9){\rotatebox{90}{magnitude of coefficients}}
  \end{overpic}
  \end{minipage}  
  \hspace{.2cm} 
    \begin{minipage}{.45\textwidth} 
 \centering
  \begin{overpic}[width=\textwidth]{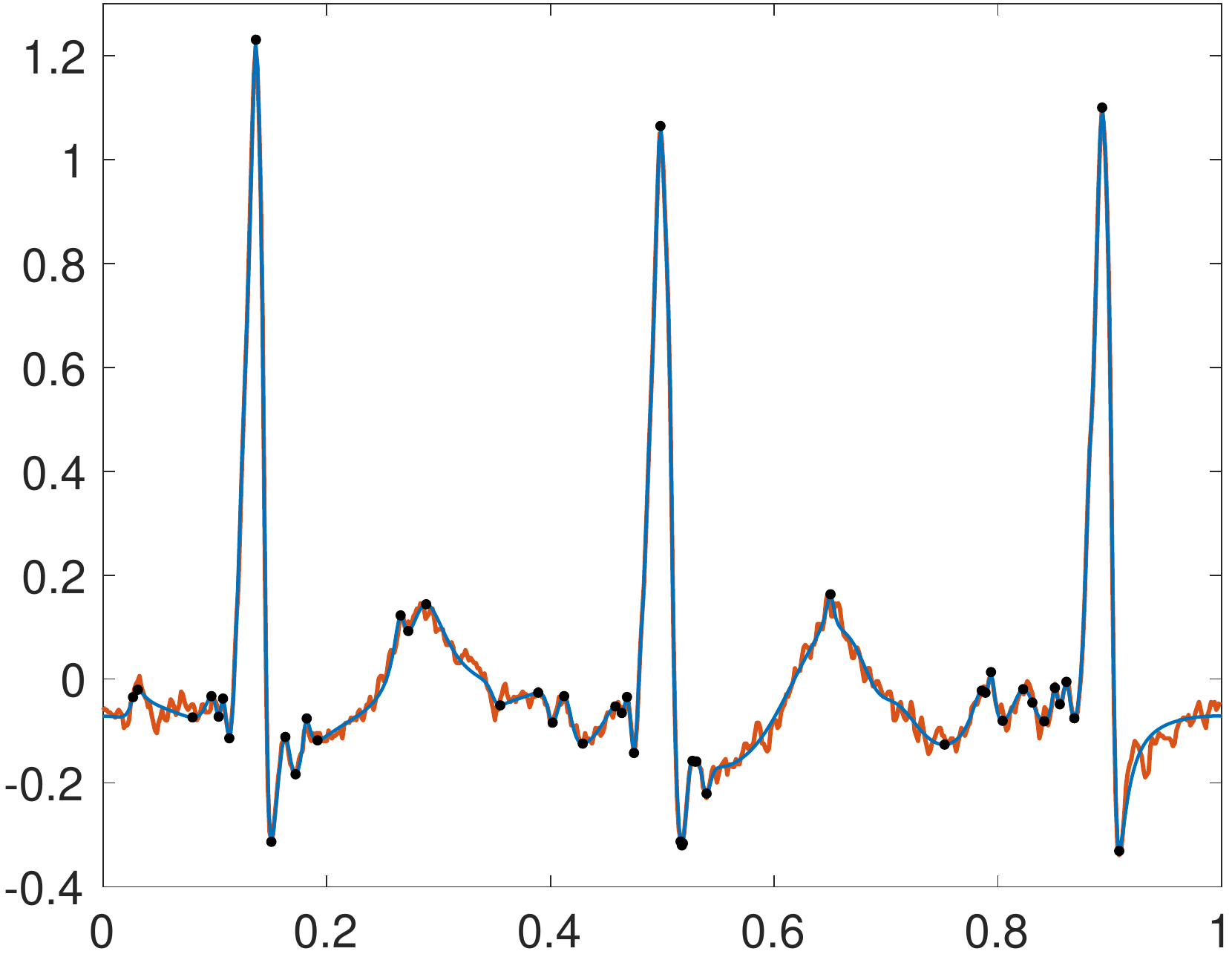}
    \put(33,80){Time domain}
      \put(-7, 9){\rotatebox{90}{amplitude of the signal}}
  \end{overpic}
  \end{minipage}
  \caption{ Left: The magnitude of the Fourier coefficients of the original signal (orange). The decay of the coefficients stagnates due to noise in the signal, and this pollutes the higher frequencies. Once an exponential sum representation $R_m$  is constructed, we can  extrapolate to higher frequencies by evaluating $R_m$,  and thereby super--resolve the signal (blue). Right: A barycentric rational approximant (blue) in the time domain is computed using the extrapolated Fourier data. It is a denoised version of the original ECG signal (orange). Local extrema are identified (black dots) using the differentiation and rootfinding algorithms from \cref{sec:REfit}. 
  }
  \label{fig:ECG}
  \end{figure}


\section{Algorithms for computing with rationals and exponential sums} \label{sec:REfit}
This section gives an overview of several of the algorithms used in our software~\cite{wilberREfit} to compute with trigonometric rationals and exponential sums. The Fourier and inverse Fourier transform functions are used to move between representations as needed. For operations on trigonometric rationals that return trigonometric rationals, we represent the resulting function using exponential sums or barycentric forms when possible. 

\subsection{Compression for suboptimal sums of exponentials} \label{sec:comp}  Exponential sums are closed under addition and multiplication, but a sum $R_{m}$ resulting from the naive application of these and other operations is often suboptimal in the sense that a shorter sum $\tilde{R}_{\tilde{m} }$ exists, where $|R_{m} (j) - \tilde{R}_{\tilde{m} }(j) | < \epsilon$ for $j \geq 0$.  A major advantage of the exponential sum format is that  $\tilde{R}_{\tilde{m} }$ can be constructed with a cost that usually depends on $m$, rather than the $\epsilon$-resolution parameter $N_\epsilon$ associated with $R_m$. 
Using AAK theory for finite rank Hankel operators, one can show (see~\cite[Thm.~3.2]{plonka2016application}) that there is a length \smash{$\tilde{m} \leq m$} approximation that satisfies
$$\|\mathcal{F}^{-1}(R_m) - \mathcal{F}^{-1}(\tilde{R}_{\tilde{m}})\|_{L_2} \leq  2\,  \sigma_{\tilde{m}}(\Gamma_R),$$ where $\Gamma_R$ is the infinite matrix with entries $(\Gamma_R)_{j + k} = R_m(j+k)$, $j, k \geq 0$.   In~\cite{plonka2016application},  an $\mathcal{O}(m^3)$ algorithm for recovering $\tilde{R}_{\tilde{m}}$ directly from the parameters of $R_{m}$ is developed; a closely-related approach using only properties of finite-dimensional Hankel matrices is described in~\cite{beylkin2005approximation}.  This method is successfully employed in~\cite{haut2012fast} within a scheme that uses rational approximations to solve Burger's equation. However, the implementation requires the judicious use of high-precision arithmetic, which we wish to avoid. 

 Instead, we note that when a length $\tilde{m} < m$ recurrence is approximately satisfied by the sequence $\{ R_m(0), R_m(1), \ldots\}$, this fact is often captured well with a modified Prony's method that involves only a small sample of $2M +1$ observations of $R_m$, where $M >m$. Specifically, we construct a small \hbox{$(M\! +\!1) \times (m \!+ \!1)$} rectangular Hankel matrix $H$ with entries $H_{j k} = R_m(j + k)$.  Then, we apply the RPM from~\cref{alg:apm} on $H$ to construct $\tilde{R}_{\tilde{m}}$. We check the error $|R_m (j) - \tilde{R}_{\tilde{m}}(j)|$ on a random sample of integers $0 \leq j \leq N$, where $N$ is the $\epsilon$-resolution parameter used in the original construction of $R_m$. When the error is too large, we increase $M$ and try again. The cost to compute $\tilde{R}_{\tilde{m} }$ is $\mathcal{O}(M m^2)$. In a worst-case scenario, $M$ can grow as large as $N$. We observe experimentally that this approach is often very effective, but more work is needed to understand the conditions under which it is guaranteed that $M \ll N$.
 
 \subsection{Sums of trigonometric rationals} If $S_\ell$ and $G_n$ are exponential sums, then  $R_m = S_\ell + G_n$ can be constructed straightforwardly. However, $R_m$ may be of suboptimal length.  We apply the compression algorithm with $\epsilon \approx \epsilon_{mach}$ to $R_m$, where $m = \ell + n$,  to find $\tilde{R}_{\tilde{m}}$. This ``compression--plus" method is especially useful for tasks that involve repeated summations and require many recompressions. The compression--plus algorithm is applied  in REfit when the `\texttt{+}' operator is used between efun objects. For summing trigonometric rationals $s_\ell$ and $g_n$ represented as rfuns, we simply evaluate the sum and then apply pronyAAA to find $r_m = s_\ell + g_n$.\footnote{If this proves difficult due to spurious poles, we use the Fourier and inverse Fourier transforms to convert to efuns, perform the addition, and then convert back to an rfun.} The rfun and efun objects can also be combined in various ways. The syntax  \hbox{\texttt{r = s + g}} adds two rfuns and returns an rfun by default. The expression \hbox{\texttt{[r, R] = s + g} }automatically retrieves the efun \hbox{\texttt{R = ft(r)}} in addition to \texttt{r}. When an rfun and efun are summed together, both rfun and efun outputs are returned. 

\subsection{Convolutions of trigonometric rationals}
\begin{figure}
\centering
    \begin{minipage}{.45\textwidth} 
 \centering
  \begin{overpic}[width=\textwidth]{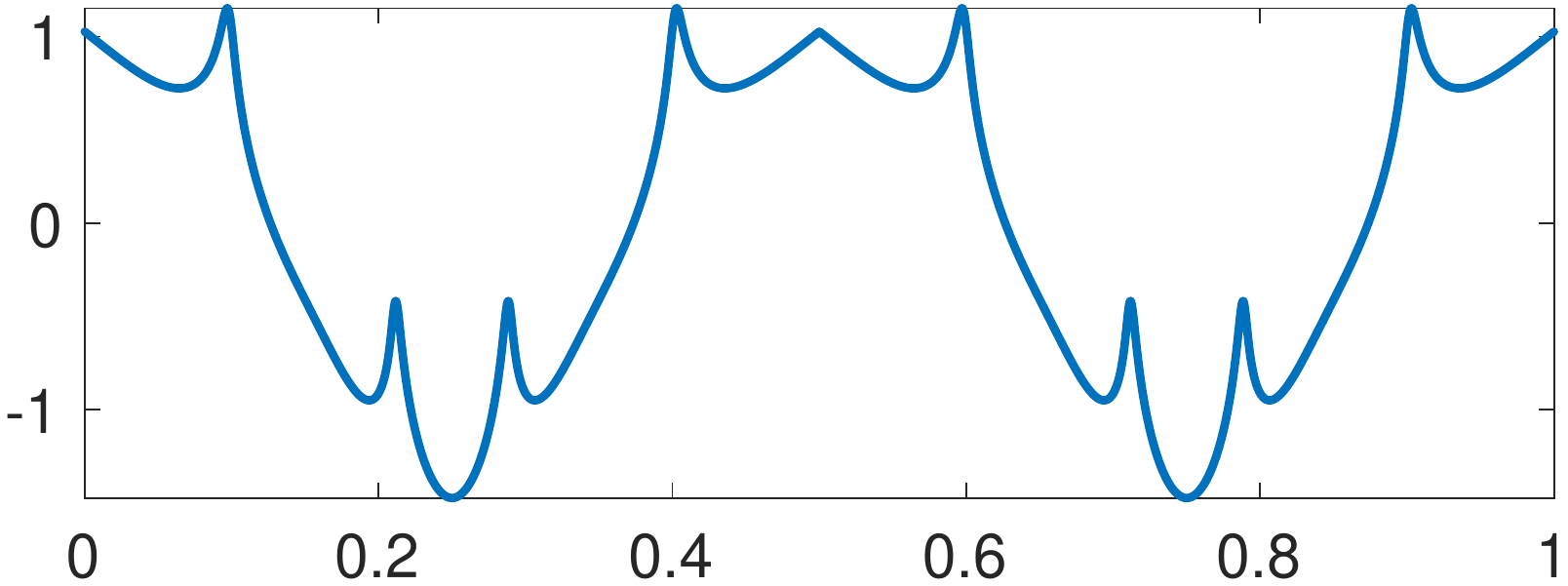}
    \put(50, -5){\rotatebox{0}{\small{$x$}}}
  \end{overpic}
  
  \vspace{.5cm}
    \begin{overpic}[width=\textwidth]{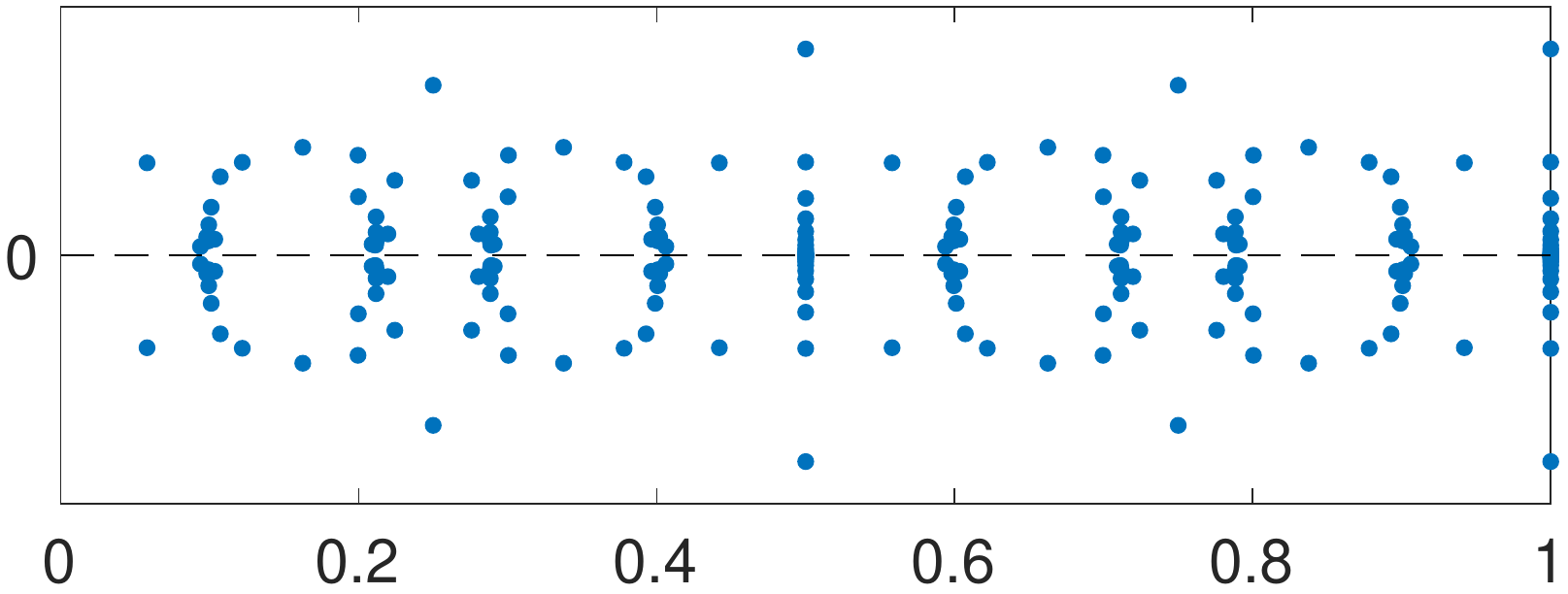}
    \put(50, -5){\rotatebox{0}{\small{$Re$}}}
    \put(-5, 15){\rotatebox{90}{\small{$Im$}}}
  \end{overpic}
  \end{minipage}
    \hspace{.2cm}
    \begin{minipage}{.45\textwidth} 
 \centering
  \begin{overpic}[width=\textwidth]{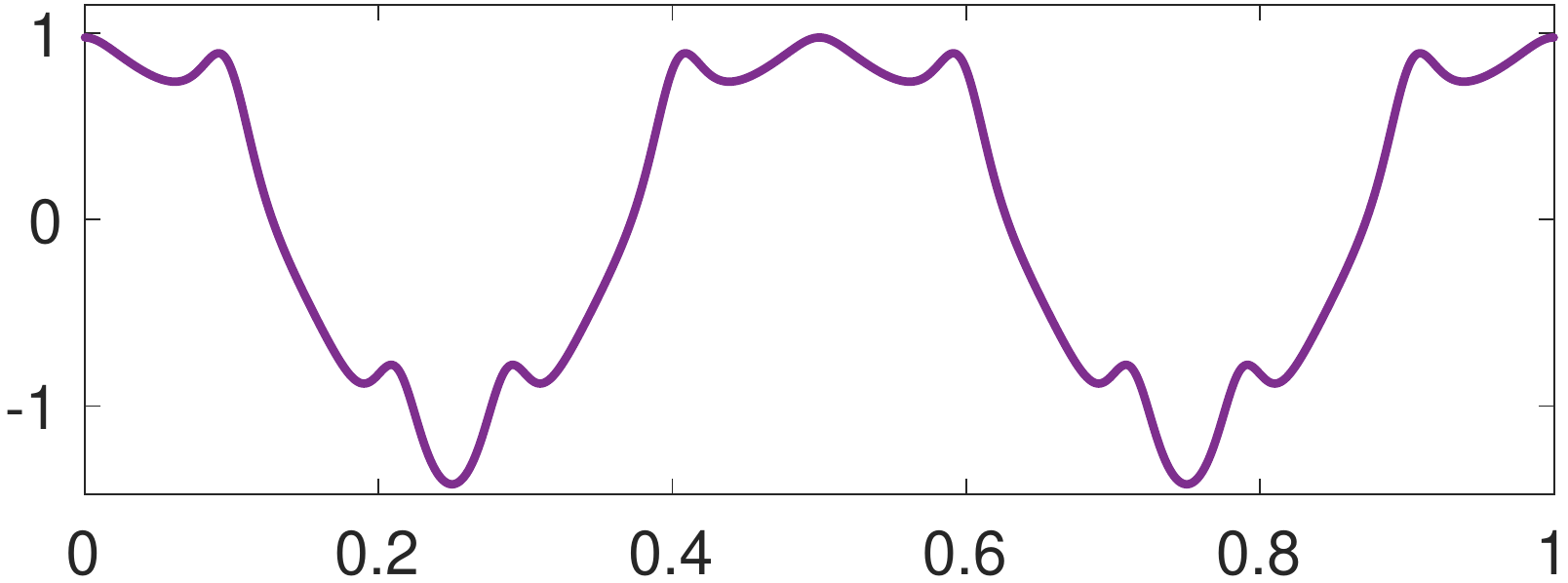}
    \put(50, -5){\rotatebox{0}{\small{$x$}}}
  \end{overpic}
  
  \vspace{.5cm}
    \begin{overpic}[width=\textwidth]{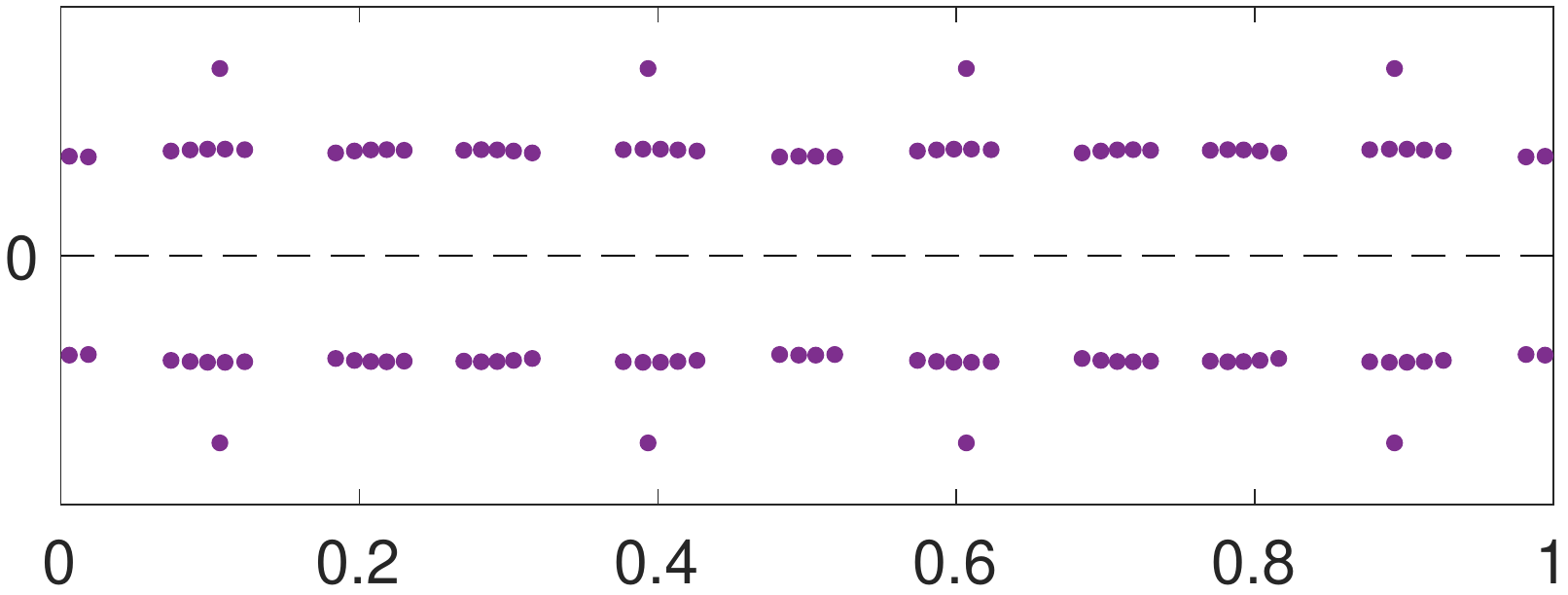}
     \put(50, -5){\rotatebox{0}{\small{$Re$}}}
      \put(-5, 15){\rotatebox{90}{\small{$Im$}}}
  \end{overpic}
  \end{minipage}
  \vspace{.3cm}
  \caption{Left (top): an exponential sum $S_\ell$ representing a function $f$ with singularities (inspired by Chebfun's ``wild" function~\cite{Chebfun}) is constructed as an efun, and $s_\ell = \mathcal{F}^{-1}(S_\ell) \approx f$ is plotted in the time domain. Here, $\ell = 116$ and $\|s_\ell-f\|_\infty \approx 10^{-9}\|f\|_\infty$.  Left (bottom): Locations of poles of $s_\ell$ occurring in a small strip of width $.2$ around the real line are plotted in the complex plane. Right (top): The convolution  $r_m = s_\ell \ast g_n $ is computed ($m = 53$), stored as an efun, and plotted in the time domain. Here, $g_n$ ($n = 13$) is a trigonometric rational approximation to a narrow Gaussian. Right (bottom): The poles of $r_m$ still cluster near the singularities of $f$, but less tightly, reflecting the fact that $f \ast g_n$ is a smoothed version of $f$.} 
  \label{fig:convolution}
  \end{figure}
  
The convolution of trigonometric rationals $s_\ell$ and $g_n$ (denoted $s_\ell \ast g_n$) can be constructed in Fourier space by finding the exponential sum $R_{m_0} = (\mathcal{F}s_\ell)( \mathcal{F} g_n) = S_\ell G_n$. The product can be computed directly in a closed form,  but this results in a large sum with $m_0 = \ell n$.  To find a shorter sum $\tilde{R}_{\tilde{m}}$, we first find an upper bound $m$ on $\tilde{m}$ by determining how many terms in $R_{m_0}$ have a negligibly small influence. We also check the decay in the tail of the first $m_0$ Fourier coefficients given by $S_\ell G_n$ to determine if even fewer samples are needed. Then, we apply the compression algorithm using  rectangular Hankel matrices of the form $H_{jk}  = S_\ell(j+k)G_n(j+k)$. For efuns, this operation is accessed  by typing \texttt{S.*G}. For rfuns, the command $\texttt{conv(s,g)}$ uses the Fourier and inverse Fourier transform functions to apply the above scheme. \cref{fig:convolution}  displays an example where $s_\ell$ ($\ell = 116$) is an approximation to a function $f$ with several singularities, and $g_n$  ($n = 13$) is an approximation to the normalized Gaussian \smash{$g(x) = \tfrac{1}{\sqrt{2 \pi} \sigma}  e^{-x^2/2\sigma^2}-1$}, with $\sigma = .01$. Our convolution algorithm constructs $r_m \approx f \ast g$, a smoothed version of $f$. The method automatically chooses $m = 53$, about half the degrees of freedom required for approximating $f$ with $s_\ell$. Little accuracy is lost in this process: $\|R_m - S_\ell G_n\|_\infty \approx 10^{-11}\|S_\ell G_n\|_\infty$. 


\subsection{Products of trigonometric rationals}  The product of two trigonometric rationals $r_\ell$ and $s_n$ in the time domain is equivalent to their convolution in Fourier space. If $R_\ell = \mathcal{F}(r_\ell)$ and $S_n = \mathcal{F}(s_n)$ are each sums of complex exponentials, then 
\begin{equation} 
\label{eq:disc_conv}
\mathcal{F}(r_\ell s_n)(k) = (R_\ell \ast S_n)(k) = \sum_{j = -\infty}^{\infty} R_\ell(k-j) S_n(j) \approx \sum_{j = -N_\epsilon}^{N_\epsilon}   R_\ell(k-j) S_n(j), 
\end{equation}
 where $N_\epsilon$ is the $\epsilon$-resolution parameter for $S_n$. The evaluation of~\cref{eq:disc_conv} at $M$ consecutive points requires a matrix-vector multiply with a Toeplitz matrix.  Since $r_\ell s_m$ is a trigonometric rational with at most $2\ell + 2n$ poles, we apply the compression algorithm to  find $G_{m} \approx \mathcal{F}(r_\ell s_n)(k)$, with $m = \ell + n$.  If \texttt{R} and \texttt{S} are  efuns, this command is accessed by typing \texttt{conv(R,S)}. If $r_\ell$ and $s_n$ are represented with rfuns \texttt{r} and \texttt{s}, respectively, then \texttt{r.*s} simply applies pronyAAA to the function $r_\ell(x)s_n(x)$to construct a new rfun representing $r_\ell s_n$.\footnote{If this proves difficult due to spurious poles, we use the Fourier and inverse Fourier transforms to convert to efuns, perform the convolution in Fourier space, and then convert back to an rfun.}  

\subsection{Differentiation} 
The $k$th derivative of $r_m$, denoted $ r_m^{(k)}$, is a trigonometric rational of type \hbox{$(k m\!-\!1, k m)$}. However, $ r_m^{(k)}$ is fundamentally of a different form than the trigonometric rationals constructed via pronyAAA and the RPM. It has $m$ conjugate pairs of poles, and each pole is of multiplicity $k$. It is possible to represent derivatives with trigonometric rationals with simple poles, but it isn't always sensible. By default, REfit returns a function handle for evaluating derivatives (or their Fourier transforms) whenever  \texttt{diff($\cdot$,k)} is applied to an rfun (or an efun). However, one can also use  \texttt{diff($\cdot$,k,`type')}, to specify that an efun or rfun should be returned.   

\paragraph{Differentiation in Fourier space} When $r_m$ is represented by the complex exponential sum $R_m$ in Fourier space,  the Fourier coefficients of \smash{$ r_m^{(k)}$} are given by  
\smash{$\mathcal{F}(r_m^{(k)}) (j) = (2 \pi i j)^k R_m(j).$}
The command \texttt{h=diff(R,k)} by default returns a handle for evaluating this function in Fourier space. If instead, for example,  one types \texttt{diff(R,k,`efun')},  the RPM is applied to construct a representation of \smash{$\mathcal{F}(r_m^{(k)}) $} as a sum of weighted complex exponentials (without polynomial coefficients). 

\paragraph{Differentiation in the time domain} Derivatives of barycentric trignometric rational interpolants  satisfy a recurrence relation and can be expressed in a simple closed form. To see this, consider the linearization of $r_m^{\gamma, t}= n_{m-1}/d_m$, which can be differentiated as $(r_m^{\gamma, t} d_m)^{\prime} = (n_{m-1})^\prime$.  Plugging in the definitions from~\cref{eq:BarycentricRat} results in the following formula, which holds everywhere on $[0, 1)$ except at the interpolating points: 
\begin{equation}
\label{eq:derivative}
(r_m^{\gamma, t})^\prime (x) = -\pi \dfrac{ \sum_{j = 1}^{2m}  \gamma_j \csc^2(\pi x - \pi t_j) \left( f_j - r_m^{\gamma, t}(x)\right)  }{\sum_{j = 1}^{2m} \gamma_j \cot( \pi x - \pi t_j)}.
\end{equation} 
To evaluate $(r_m^{\gamma, t})^\prime$ at the interpolating points $t = (t_1, \ldots, t_{2m})^T$, we use the 
special differentiation matrices introduced in~\cite{baltensperger2002some}. Explicit descriptions of recursive formulas for computing higher derivatives are also found in~\cite{baltensperger2002some}. All of this is encoded within a function handle that is accessed in REfit by applying the command $\texttt{diff}$ to an rfun.

\subsection{Integration}  The indefinite integral of a trigonometric rational $r_m$ is not itself a trigonometric rational.  In Fourier space, if $R_m = \mathcal{F}(r_m)$ is a sum of complex exponentials as in~\cref{eq:Exp_Sum}, then except at $k = 0$, the Fourier coefficients of $\mathcal{F}(g)$, where 
$ g(y) = \int_{0}^y r_m(x) dx,$
are given by $\hat{g}_{k} = R_m(k)/2 \pi i k$.  A function handle  for evaluating $\mathcal{F}(g)$ is returned when \texttt{cumsum} is applied to an efun. One can also fit a new complex exponential to $\mathcal{F}(g)$ by typing \texttt{cumsum($\cdot$,`efun')}, though it may not be an efficient representation.  

If \texttt{cumsum} is applied to an rfun, we supply a handle for $g$ that applies Gauss-Legendre quadrature~\cite{hale2013fast}.  The stable evaluation property of the barycentric form is advantageous here, which is why we do not instead make use of the pole-residue form of the rational $r_m(z)$ in~\cref{eq:zt_form}. To integrate $r_m$ over a finite interval $[a, b] \subset [0, 1)$, the command \texttt{sum($\cdot$,a,b)} can be applied to an rfun or an efun. 

\subsection{Rootfinding and polefinding}
\label{sec:Roots}
The roots of the barycentric trigonometric rational $r_m^{\gamma, t}$ coincide with the generalized eigenvalues of a matrix pencil. Specifically, if $r_m^{\gamma, t}(\zeta_j) = 0$ and $\mu = e^{ 2 \pi i \zeta_j}$, then there is  nonzero $y$ such that $E y = \mu By$, where 
\begin{equation}
E = \left[ \begin{array}{ ccc|c }
 e^{2\pi i x_1} &  &&  i \omega_1e^{2 \pi i x_1} \\
 &  \ddots & &  \vdots \\
  &  &  e^{2\pi i  x_{2m}} & i \omega_{2m}e^{2 \pi i x_{2m}}\\ 
  \hline
f_1  & \cdots & f_{2m}& 0
\end{array}\right], \quad 
B = \left[ \begin{array}{ ccc|c }
 1&  &&  i \omega_1 \\
 &  \ddots & &  \vdots \\
  &  &  1 & i \omega_{2m}\\
  \hline
0  & \cdots &0 & 0
\end{array}\right].
\end{equation}
The pencil $(E, B)$ has at least two infinite eigenvalues and  one eigenvalue at  $\mu_{0} = 0$ corresponding to $\zeta_0 = -\infty$ (this captures the asymptotic behavior of $r_m^{\gamma, t}$). Once the remaining $2m-2$ eigenvalues are found, the  zeros of $r_m^{\gamma, t}$ are immediate. The command \texttt{roots} applied to rfuns or efuns applies this algorithm and returns real-valued roots. For efuns, this requires first converting to an rfun via the Fourier transform. The command \texttt{roots($\cdot$,`all')} additionally returns complex-valued roots. 

The poles of $r_m^{\gamma, t}$ can be found in a similar way: the pencil $(\tilde{E}, B)$, where $\tilde{E}$ is identical to $E$ except that each $f_j$ in the last row is replaced by $1$, has at least one infinite eigenvalue. If $\tilde{\mu}_j$ is one of the remaining finite eigenvalues, then $ \eta_j = \log \tilde{\mu}_j / 2 \pi i $ is a pole of $r_m^{\gamma, t}$.
If it is important to preserve the pole symmetry, it is better to represent $r_m$ with an exponential sum $R_m$.

\paragraph{Residues} We compute the residues of the barycentric interpolant $r_m^{\gamma, t}$ using the fact that $r_m^{\gamma, t} = n_{m-1}/d_m$, where $n_{m-1}$ and $d_m$ are trigonometric polynomials as in~\cref{eq:BarycentricRat}. Since the poles of $r_m^{\gamma, t}$ are simple, the residue at the pole $\eta_j$ is given by
$ {\rm Res}(r_m^{\gamma, t}, \eta_j) = n_{m-1}(\eta_j)/ d_m^{ \, \prime} (\eta_j).$
The residues of the poles of $\mathcal{F}^{-1}(R_m)$, where $R_m$ is an exponential sum, have a closed form formula (see \cref{sec:FT}). 

\paragraph{Minima and Maxima}
The commands \texttt{min} or \texttt{max} a return the local minima (or maxima) attained by the represented trigonometric rational on the interval $[0, 1)$. The global minimum, for example, can be found by typing \texttt{min(min($\cdot$))}. 
To compute the extrema,  we use a rfun and apply the differentiation formula in~\cref{eq:derivative} to evaluate its derivative.  We use this to construct a rfun representing $(r_m^{\gamma, t})^\prime$,  find its roots, and then test for concavity. For an application, see \cref{sec:ECG}.  

\subsection{Additional commands} The REfit package includes commands for data visualization and common tasks in signal processing, such as filtering and cross-correlation. Commands related to the pole--residue format of the rational $\tilde{r}_m(z)$ from~\cref{eq:zt_form} are also available, as this format is closely related to the notion of the $z$-transform~\cite{oppenheim2001discrete} and is useful for interpretation and analysis. 

\subsection{Nonperiodic representations} \label{sec:nonper}
It is natural to consider the extension of these ideas and algorithms to non-periodic signals.  In signal space, a slight modification of the standard AAA algorithm can be used to construct type $(k\!-\!1, k)$ barycentric rational interpolants with basis functions as in~\eqref{eq:std_bary}. A type $(k\!-\!1, k)$ rational function with simple poles that have nonzero imaginary parts can be associated with a Laurent expansion, and the Laurent coefficients can be expressed in terms of two short sums of complex exponentials~\cite[Ch.~4]{stein2010complex}. The method in~\cref{sec:FT} can be used to find the parameters of the exponential sums, though one requires a scheme for computing the Laurent coefficients accurately from data samples.
Modifications of the algorithms in~\cref{sec:REfit} can then be applied to compute with Laurent coefficients and the associated rational barycentric interpolants.
 
A general approach favored by many practitioners is to apply windowing strategies~\cite{gabor1946theory, gottlieb1997gibbs, harris1978use} to compute the Fourier transform of the non-periodic signal, or to use Fourier extension methods~\cite{bruno2003fast}. If the signal is first represented by a barycentric rational interpolant using AAA, then the pole or barycentric node locations can be used to adaptively select parameters for determining windows and designing quadrature schemes.  We hope in future work to develop these observations into  strategies for computing with time-frequency representations of general rational functions. 

\section{Conclusion} We have introduced a framework for signal reconstruction and automated computing that employs efficient representations in both time and frequency space. Our work integrates ideas from the harmonic analysis community involving exponential sums and Hankel operator theory~\cite{beylkin2005approximation, plonka2016application} with developments in adaptive barycentric rational interpolation~\cite{berrut1988rational, henrici1979barycentric, nakatsukasa2018aaa}.  An implementation of all of the described methods is publicly available in the REfit software package~\cite{wilberREfit}. 

\section*{Acknowledgements} We thank Marc Aur\`ele Gilles and Nick Trefethen for reading an early draft of this manuscript. We also thank Yuji Nakatsukasa and Nick Trefethen for several valuable conversations, and we thank two anonymous referees for their helpful suggestions.

\bibliographystyle{siam}
\bibliography{refs}

\begin{thebibliography}{10}

\bibitem{aurentz2017chopping}
{\sc J.~L. Aurentz and L.~N. Trefethen}, {\em Chopping a {C}hebyshev series},
  acmtoms, 43 (2017), pp.~1--21.

\bibitem{austin2017numerical}
{\sc A.~P. Austin and K.~Xu}, {\em On the numerical stability of the second
  barycentric formula for trigonometric interpolation in shifted equispaced
  points}, IMA J. Numer. Anal., 37 (2017), pp.~1355--1374.

\bibitem{baddoo2020aaatrig}
{\sc P.~J. Baddoo}, {\em The {AAA}trig algorithm for rational approximation of
  periodic functions}, SIAM J. Sci. Comput., 43 (2021), pp.~A3372--A3392.

\bibitem{baltensperger2002some}
{\sc R.~Baltensperger}, {\em Some results on linear rational trigonometric
  interpolation}, Comp. Math. Appl., 43 (2002), pp.~737--746.

\bibitem{berljafa2017rkfit}
{\sc M.~Berljafa and S.~G{\"u}ttel}, {\em The {RKFIT} algorithm for nonlinear
  rational approximation}, SIAM J. Sci. Comput., 39 (2017), pp.~A2049--A2071.

\bibitem{berrut1988rational}
{\sc J.~Berrut}, {\em Rational functions for guaranteed and experimentally
  well-conditioned global interpolation}, Comp. Math. Appl., 15 (1988),
  pp.~1--16.

\bibitem{berrut2005recent}
{\sc J.-P. Berrut, R.~Baltensperger, and H.~D. Mittelmann}, {\em Recent
  developments in barycentric rational interpolation}, in Trends and
  Applications in Constructive Approximation, Springer, 2005, pp.~27--51.

\bibitem{berrut2004barycentric}
{\sc J.-P. Berrut and L.~N. Trefethen}, {\em Barycentric {L}agrange
  interpolation}, SIAM Rev., 46 (2004), pp.~501--517.

\bibitem{beylkin2005approximation}
{\sc G.~Beylkin and L.~Monz{\'o}n}, {\em On approximation of functions by
  exponential sums}, Appl. and Comp. Harmonic Analysis, 19 (2005), pp.~17--48.

\bibitem{beylkin2009nonlinear}
\leavevmode\vrule height 2pt depth -1.6pt width 23pt, {\em Nonlinear inversion
  of a band-limited {F}ourier transform}, Appl. and Comp. Harmonic Analysis, 27
  (2009), pp.~351--366.

\bibitem{bruno2003fast}
{\sc O.~P. Bruno}, {\em Fast, high-order, high-frequency integral methods for
  computational acoustics and electromagnetics}, in Topics in computational
  wave propagation, Springer, 2003, pp.~43--82.

\bibitem{candes2014towards}
{\sc E.~J. Cand{\`e}s and C.~Fernandez-Granda}, {\em Towards a mathematical
  theory of super-resolution}, Comm. Pure Appl. Math., 67 (2014), pp.~906--956.

\bibitem{vcerveny1982computation}
{\sc V.~{\v{C}}erven{\`y}, M.~M. Popov, and I.~P{\v{s}}en{\v{c}}{\'\i}k}, {\em
  Computation of wave fields in inhomogeneous media—{G}aussian beam
  approach}, Geophys. J. Int., 70 (1982), pp.~109--128.

\bibitem{cichowicz1993automatic}
{\sc A.~Cichowicz}, {\em An automatic s-phase picker}, Bulletin of the
  Seismological Society of America, 83 (1993), pp.~180--189.

\bibitem{costa2021aaa}
{\sc S.~Costa and L.~N. Trefethen}, {\em {AAA}-least squares rational
  approximation and solution of {L}aplace problems}, arXiv preprint
  arXiv:2107.01574,  (2021).

\bibitem{de1972calculating}
{\sc C.~De~Boor}, {\em On calculating with {B}-splines}, J. Approx. Theory, 6
  (1972), pp.~50--62.

\bibitem{debnath2003wavelets}
{\sc L.~Debnath}, {\em Wavelets and signal processing}, Springer Science \&
  Business Media, 2003.

\bibitem{derevianko2021exact}
{\sc N.~Derevianko and G.~Plonka}, {\em Exact reconstruction of extended
  exponential sums using rational approximation of their {F}ourier
  coefficients}, arXiv preprint arXiv:2103.07743,  (2021).

\bibitem{Chebfun}
{\sc T.~A. Driscoll, N.~Hale, and L.~N. Trefethen}, eds., {\em Chebfun Guide},
  Pafnuty Publications, Oxford, 2014.

\bibitem{elsworth2019conversions}
{\sc S.~Elsworth and S.~G{\"u}ttel}, {\em Conversions between barycentric,
  {RKFUN}, and {N}ewton representations of rational interpolants}, Lin. Alg.
  Appl., 576 (2019), pp.~246--257.

\bibitem{filip2018rational}
{\sc S.-I. Filip, Y.~Nakatsukasa, L.~N. Trefethen, and B.~Beckermann}, {\em
  Rational minimax approximation via adaptive barycentric representations},
  SIAM J. Sci. Comput., 40 (2018), pp.~A2427--A2455.

\bibitem{francis1987course}
{\sc B.~Francis}, {\em A course in {H}-infinity control theory}, Lecture notes
  in Control and Information Sciences, 88 (1987), p.~R5.

\bibitem{fridli2011rational}
{\sc S.~Fridli, L.~L{\'o}csi, and F.~Schipp}, {\em Rational function systems in
  {ECG} processing}, in Int. Conf. Comp. Aided Sys. Theory, Springer, 2011,
  pp.~88--95.

\bibitem{gabor1946theory}
{\sc D.~Gabor}, {\em Theory of communication. part 1: The analysis of
  information}, J. Inst. of Electrical Engineers-Part III: Radio and Comm.
  Engineering, 93 (1946), pp.~429--441.

\bibitem{gilian2014ecg}
{\sc Z.~Gili{\'a}n}, {\em {ECG}-based heart beat detection using rational
  functions}, in Conf. Dyadic Anal. Rel. Fields Appl., vol.~13, 2014.

\bibitem{golub2012matrix}
{\sc G.~H. Golub and C.~F. Van~Loan}, {\em Matrix Computations}, vol.~3, John
  Hopkins University Press, 2012.

\bibitem{gottlieb1997gibbs}
{\sc D.~Gottlieb and C.-W. Shu}, {\em On the {G}ibbs phenomenon and its
  resolution}, SIAM Rev., 39 (1997), pp.~644--668.

\bibitem{gu1996efficient}
{\sc M.~Gu and S.~C. Eisenstat}, {\em Efficient algorithms for computing a
  strong rank-revealing {QR} factorization}, SIAM J. Sci. Comput., 17 (1996),
  pp.~848--869.

\bibitem{gustavsen1999rational}
{\sc B.~Gustavsen and A.~Semlyen}, {\em Rational approximation of frequency
  domain responses by vector fitting}, {IEEE} Trans. Power Deliv., 14 (1999),
  pp.~1052--1061.

\bibitem{hale2013fast}
{\sc N.~Hale and A.~Townsend}, {\em Fast and accurate computation of
  {G}auss--{L}egendre and {G}auss--{J}acobi quadrature nodes and weights}, SIAM
  J. Sci. Comput., 35 (2013), pp.~A652--A674.

\bibitem{halko2011finding}
{\sc N.~Halko, P.-G. Martinsson, and J.~A. Tropp}, {\em Finding structure with
  randomness: Probabilistic algorithms for constructing approximate matrix
  decompositions}, SIAM Rev., 53 (2011), pp.~217--288.

\bibitem{harris1978use}
{\sc F.~J. Harris}, {\em On the use of windows for harmonic analysis with the
  discrete {F}ourier transform}, Proc. {IEEE}, 66 (1978), pp.~51--83.

\bibitem{haut2012fast}
{\sc T.~Haut and G.~Beylkin}, {\em Fast and accurate con-eigenvalue algorithm
  for optimal rational approximations}, SIAM J. Matrix Anal. Appl., 33 (2012),
  pp.~1101--1125.

\bibitem{haut2013solving}
{\sc T.~Haut, G.~Beylkin, and L.~Monz{\'o}n}, {\em Solving {B}urgers' equation
  using optimal rational approximations}, Appl. Comp. Harm. Appl., 34 (2013),
  pp.~83--95.

\bibitem{henrici1979barycentric}
{\sc P.~Henrici}, {\em {B}arycentric formulas for interpolating trigonometric
  polynomials and their conjugates}, Num. Math., 33 (1979), pp.~225--234.

\bibitem{higham2004numerical}
{\sc N.~J. Higham}, {\em The numerical stability of barycentric {L}agrange
  interpolation}, IMA J. Numer. Anal., 24 (2004), pp.~547--556.

\bibitem{javed2016algorithms}
{\sc M.~Javed}, {\em Algorithms for trigonometric polynomial and rational
  approximation}, PhD thesis, University of Oxford, 2016.

\bibitem{karachalios2021loewner}
{\sc D.~Karachalios, I.~V. Gosea, and A.~C. Antoulas}, {\em The {L}oewner
  framework for system identification and reduction}, in Model Reduction
  Handbook: Volume I: System-and Data-Driven Methods and Algorithms, De
  Gruyter, 2021, pp.~181--228.

\bibitem{kay1993fundamentals}
{\sc S.~M. Kay}, {\em Fundamentals of statistical signal processing}, Prentice
  Hall PTR, 1993.

\bibitem{kovacs2019generalized}
{\sc P.~Kov{\'a}cs, S.~Fridli, and F.~Schipp}, {\em Generalized rational
  variable projection with application in {ECG} compression}, {IEEE} Trans.
  Sig. Proc., 68 (2019), pp.~478--492.

\bibitem{martin2021linear}
{\sc E.~R. Martin}, {\em A linear algorithm for ambient seismic noise double
  beamforming without explicit crosscorrelations}, Geophysics, 86 (2021),
  pp.~F1--F8.

\bibitem{MATLABsig}
{\sc Mathworks}, {\em {MATLAB} Signal Processing Toolbox: (v R2020b)},
  MathWorks, 2020.

\bibitem{miller1970stabilized}
{\sc K.~Miller}, {\em Stabilized numerical analytic prolongation with poles},
  SIAM J. Appl. Math., 18 (1970), pp.~346--363.

\bibitem{moody2000physionet}
{\sc G.~Moody, R.~Mark, and A.~Goldberger}, {\em Physionet: A research resource
  for studies of complex physiologic and biomedical signals}, in Computers in
  Cardiology 2000. Vol. 27 (Cat. 00CH37163), {IEEE}, 2000, pp.~179--182.

\bibitem{nakatsukasa2018aaa}
{\sc Y.~Nakatsukasa, O.~S{\`e}te, and L.~Trefethen}, {\em The {AAA} algorithm
  for rational approximation}, SIAM J. Sci. Comput., 40 (2018),
  pp.~A1494--A1522.

\bibitem{oppenheim2001discrete}
{\sc A.~V. Oppenheim, J.~R. Buck, and R.~W. Schafer}, {\em Discrete-time signal
  processing. Vol. 2}, Upper Saddle River, NJ: Prentice Hall, 2001.

\bibitem{plonka2016application}
{\sc V.~Pototskaia and G.~Plonka}, {\em Application of the {AAK} theory and
  {P}rony-like methods for sparse approximation of exponential sums}, Proc.
  Appl. Math. Mech., 17 (2017), pp.~835--836.

\bibitem{potts2011nonlinear}
{\sc D.~Potts and M.~Tasche}, {\em Nonlinear approximation by sums of
  nonincreasing exponentials}, Applicable Anal., 90 (2011), pp.~609--626.

\bibitem{potts2013parameter}
\leavevmode\vrule height 2pt depth -1.6pt width 23pt, {\em Parameter estimation
  for nonincreasing exponential sums by {P}rony-like methods}, Lin. Alg. Appl.,
  439 (2013), pp.~1024--1039.

\bibitem{prony1795essai}
{\sc G.~Prony}, {\em Essai experimental et analytique}, J. de l’Ecole
  Polytechnique, 2 (1795).

\bibitem{ricker2012transient}
{\sc N.~Ricker}, {\em Transient waves in visco-elastic media}, vol.~10,
  Elsevier, 2012.

\bibitem{rippa1999algorithm}
{\sc S.~Rippa}, {\em An algorithm for selecting a good value for the parameter
  c in radial basis function interpolation}, Adv. in Comp. Math., 11 (1999),
  pp.~193--210.

\bibitem{stein2010complex}
{\sc E.~M. Stein and R.~Shakarchi}, {\em Complex analysis}, vol.~2, Princeton
  University Press, 2010.

\bibitem{transtrum2010nonlinear}
{\sc M.~K. Transtrum, B.~B. Machta, and J.~P. Sethna}, {\em Why are nonlinear
  fits to data so challenging?}, Phys. Rev. Letters, 104 (2010), p.~060201.

\bibitem{trefethen2013approximation}
{\sc L.~N. Trefethen}, {\em Approximation {T}heory and {A}pproximation
  {P}ractice}, {SIAM}, 2013.

\bibitem{trefethen2021exponential}
{\sc L.~N. Trefethen, Y.~Nakatsukasa, and J.~Weideman}, {\em Exponential node
  clustering at singularities for rational approximation, quadrature, and
  {PDE}s}, Numer. Math., 147 (2021), pp.~227--254.

\bibitem{unser2000wavelets}
{\sc M.~A. Unser and T.~Blu}, {\em Wavelets and radial basis functions: A
  unifying perspective}, in Wavelet Appl. Sig. Im. Proc. VIII, vol.~4119, Int.
  Soc. Opt. Phot., 2000, pp.~487--493.

\bibitem{vetterli2002sampling}
{\sc M.~Vetterli, P.~Marziliano, and T.~Blu}, {\em Sampling signals with finite
  rate of innovation}, {IEEE} Trans. Sig. Proc., 50 (2002), pp.~1417--1428.

\bibitem{wilberREfit}
{\sc H.~Wilber}, {\em https://github.com/heatherw3521/refit}, 2021.

\bibitem{wright2015extension}
{\sc G.~B. Wright, M.~Javed, H.~Montanelli, and L.~N. Trefethen}, {\em
  Extension of {C}hebfun to periodic functions}, SIAM J. Sci. Comput., 37
  (2015), pp.~C554--C573.

\bibitem{xu2013bootstrap}
{\sc K.~Xu and S.~Jiang}, {\em A bootstrap method for sum-of-poles
  approximations}, J. of Sci. Comp., 55 (2013), pp.~16--39.

\end{thebibliography}

\end{document}